\documentclass[aap,MSNbibl,seceqn,dvips]{arximspdf}
\usepackage{graphicx}
\doi{10.1214/12-AAP880} 
\volume{23}
\issue{4}
\pubyear{2013}
\firstpage{1544}
\lastpage{1583}

\makeatletter

\newcommand{\rright}{\right}
\newcommand{\lleft}{\left}
\newcommand{\rrvert}{\vert}
\newcommand{\llvert}{\vert}

\newtheorem{theorem}{Theorem}[section]
\newtheorem{prop}[theorem]{Proposition}
\newtheorem{lem}[theorem]{Lemma}
\newtheorem{cor}[theorem]{Corollary}

\newproclaim{remark}[theorem]{Remark}
\newproclaim{defi}[theorem]{Definition}

\newcommand{\xrightarrow}[1]{\,\mathop{\hbox to 1cm{\rightarrowfill}}_{#1}\,}
\newcommand{\xxrightarrow}[1]{\mathop{\hbox to 1cm{\rightarrowfill}}_{#1}}

\renewcommand{\subset}{\subseteq}
\newcommand{\E}{\mathbb{E}}
\newcommand{\EEo}{\mathbf{E}^\omega}
\newcommand{\ee}{\mathbf{e}}
\newcommand{\N}{\mathbb{N}}
\renewcommand{\L}{\mathcal{L}}
\newcommand{\LL}{\mathbb{L}}
\newcommand{\one}{\mathbf{1}}
\newcommand{\R}{\mathbb{R}}
\newcommand{\Z}{\mathbb{Z}}
\renewcommand{\P}{\mathbb{P}}
\newcommand{\PP}{\mathbf{P}}
\newcommand{\PPo}{\mathbf{P}^\omega}
\newcommand{\eps}{\varepsilon}
\renewcommand{\d}{{\mathrm{d}}}
\newcommand{\mfk}{\mathfrak}
\newcommand{\na}{\nabla}
\newcommand{\Ah}{A_{\mathrm{hom}}}
\newcommand{\Ahd}{A_{\mathrm{hom}}^{\mathrm{disc}}}
\newcommand{\om}{\omega}

\newcommand{\V}{\mathsf{V}}

\makeatother

\begin{document}
\begin{frontmatter}

\title{Quantitative version of the Kipnis--Varadhan theorem and Monte
Carlo approximation of homogenized coefficients\thanksref{T1}}
\runtitle{Monte Carlo approximation of homogenized coefficients}

\thankstext{T1}{Supported by INRIA through the ``Action de recherche
collaborative''
DISCO, by Ministry of Higher Education and Research,
Nord-Pas de Calais Regional Council and FEDER through the ``Contrat de
Projets \'{E}tat R\'{e}gion (CPER) 2007--2013.''}

\begin{aug}
\author[A]{\fnms{Antoine} \snm{Gloria}\corref{}\ead[label=e1]{agloria@ulb.ac.be}}
\and
\author[B]{\fnms{Jean-Christophe} \snm{Mourrat}\ead[label=e2]{jean-christophe.mourrat@epfl.ch}}
\runauthor{A. Gloria and J.-C. Mourrat}
\affiliation{INRIA and EPFL}
\address[A]{Universit\'e Libre de Bruxelles (ULB)\\
Brussels\\
Belgium\\
and\\
Project--team SIMPAF\\
Inria Lille - Nord Europe\\
Villeneuve d'Ascq\\
France\\
\printead{e1}}
\address[B]{EPFL\\
Institut de Math\'{e}matiques\\
Station 8, 1015 Lausanne\\
Switzerland\\
\printead{e2}} 
\end{aug}

\received{\smonth{11} \syear{2011}}
\revised{\smonth{4} \syear{2012}}

%
\begin{abstract}
This article is devoted to the analysis of a Monte Carlo method to
approximate effective coefficients in stochastic homogenization of
discrete elliptic equations. We consider the case of independent and
identically distributed coefficients, and adopt the point of view of
the random walk in a random environment. Given some final time $t>0$, a
natural approximation of the homogenized coefficients is given by the
empirical average of the final squared positions re-scaled by $t$ of
$n$ independent random walks in $n$ independent environments. Relying
on a quantitative version of the Kipnis--Varadhan theorem combined with
estimates of spectral exponents obtained by an original combination of
PDE arguments and spectral theory, we first give a sharp estimate of
the error between the homogenized coefficients and the expectation of
the re-scaled final position of the random walk in terms of $t$. We
then complete the error analysis by quantifying the fluctuations of the
empirical average in terms of $n$ and $t$, and prove a large-deviation
estimate, as well as a central limit theorem. Our estimates are
optimal, up to a logarithmic correction in dimension $2$.
\end{abstract}

%

\end{frontmatter}

\section{Main result and structure of the proof}\label{sec1}

\subsection{Main result}\label{sec1.1}

We consider the discrete elliptic operator $-\nabla^*\cdot A\nabla$,
where $\nabla^*\cdot$ and $\nabla$
are the discrete backward divergence and forward gradient,
respectively. For all $x \in\Z^d$, $A(x)$ is the diagonal
matrix whose entries are the conductances $\omega_{x,x+\ee_i}$ of the
edges $(x,x+\ee_i)$ starting at $x$, where $(\ee_i)_{i\in\{1,\ldots,d\}}$
denotes\vspace*{1pt} the canonical basis of $\R^d$. Let $\mathbb{B}$ denote the
set of \textit{unoriented} edges of $\Z^d$. We call the family of
conductances $\omega= (\omega_e)_{e \in\mathbb{B}}$ the \textit
{environment}. This environment is \textit{symmetric} in the sense that
for all $x,y\in\Z^d$ with $|x-y|=1$, we have
$e=(x,y)=(y,x)$, so that $\omega_{x,y}=\omega_{y,x}=\omega_e$. The
environment $\omega$ is random, and we write $\P$ for its
distribution (with corresponding expectation~$\E$). We make the
following assumptions:
\begin{longlist}[(H3)]
\item[(H1)] the measure $\P$ is invariant under translations,
\item[(H2)] the conductances are
i.i.d.,\setcounter{footnote}{1}\footnote{(H2) obviously implies (H1) in
the present form. Yet for most qualitative (and some quantitative)
results (H2) can be weakened and may not imply (H1) any longer.}
\item[(H3)] there exists $0 < \alpha< \beta$ such that $\alpha\le
\omega_e \le\beta$ almost surely.
\end{longlist}

Under these conditions, standard homogenization results ensure
that there exists some \textit{deterministic} symmetric matrix $\Ah$
such that the solution operator of the deterministic continuous
differential operator $- \nabla\cdot\Ah\nabla$ describes the large
scale behavior of the solution operator
of the random discrete differential operator $-\nabla^*\cdot A\nabla$
almost surely [for this statement, (H2) can in fact be replaced by the
weaker assumption that the measure $\P$ is ergodic with respect to the
group of translations] (see~\cite{Kunnemann-83}).

The operator $-\na^* \cdot A \na$ is the infinitesimal generator of a
stochastic process $(X(t))_{t \in\R_+}$ which can be defined as
follows. Given an environment $\omega$, it is the Markov process whose
jump rate from a site $x \in\Z^d$ to a neighboring site $y$ is given
by $\omega_{x,y}$. We write $\PPo_x$ for the law of this process
starting from $x \in\Z^d$.

It is proved in~\cite{Kipnis-Varadhan-86} that under the averaged
measure $\P\PPo_0$, the re-scaled process $\sqrt{\eps} X(\eps^{-1}
t)$ converges in law, as $\eps$ tends to $0$, to a Brownian motion
whose infinitesimal generator is $- \nabla\cdot\Ah\nabla$, or in
other words, a Brownian motion with covariance matrix $2 \Ah$ (see
also~\cite{aks,Kunnemann-83,kozlov} for prior results). We will use
this fact to construct computable approximations of $\Ah$. As proved
in~\cite{DeMasi-Ferrari-89}, this invariance principle holds as soon
as (H1) is true, (H2) is
replaced by the ergodicity of the measure $\P$ and (H3) by the
integrability of the conductances. Under the assumptions (H1)--(H3),
\cite{sid} strengthens this result in another direction, showing that
for almost every environment, $\sqrt{\eps} X(\eps^{-1} t)$ converges
in law under $\PPo_0$ to a Brownian motion with covariance matrix $2
\Ah$. This has been itself extended to environments which do not
satisfy the uniform ellipticity condition (H3) (see \cite
{berbisk,matpiat,bisk,mat,bardeu}).

Let $(Y(t))_{t \in\N}$ denote the sequence of consecutive sites
visited by the random walk $(X(t))_{t\in\R_+}$ [note
that the ``times'' are different in nature for $X(t)$ and $Y(t)$]. This
sequence is itself a Markov chain that satisfies for any two neighbors
\mbox{$x,y \in\Z^d$},
\[
\PPo_x\bigl[Y(1) = y\bigr] = \frac{\omega_{x,y}}{p_\omega(x)},
\]
where $p_\omega(x) = \sum_{|z| = 1} \omega_{x,x+z}$. We simply write
$p(\omega)$ for $p_\omega(0)$. Let us introduce a ``tilted'' version
of the law $\P$ on the environments, that we write $\tilde{\P}$ and
define by
%
%
\begin{equation}
\label{deftdP}
\d\tilde{\P}(\omega) = \frac{p(\omega)}{\E[p]} \,\d\P(\omega).
\end{equation}
The reason why this measure is natural to consider is that it makes the
environment seen from the position of the random walk $Y$ a stationary process
[see (\ref{defenvpart}) for a definition of this process].

Interpolating between two integers by a straight line, we can think of
$Y$ as a continuous function on $\R_+$. With this in mind, it is also
true that there exists a matrix $\Ahd$ such that, as $\eps$ tends to
$0$, the re-scaled
process $\sqrt{\eps} Y(\eps^{-1} t)$ converges in law under $\tilde{\P
}\PPo_0$ to a Brownian motion with covariance matrix $2 \Ahd$.
Moreover, $\Ahd$ and $\Ah$ are related by (see \cite
{DeMasi-Ferrari-89}, Theorem 4.5(ii))
%
%
\begin{equation}
\label{relAhAhd}
\Ah= \E[p] \Ahd= 2d\E[\omega_e] \Ahd.
\end{equation}

Given that the numerical simulation of $Y$ saves some operations
compared to the simulation of $X$ (there is no waiting time to compute,
and the running time is equal to the number of steps), we will focus on
approximating $\Ahd$. More precisely, we fix once and for all some
$\xi\in\R^d$ with $|\xi|=1$, and define
%
%
\begin{eqnarray}
\label{defsigmat}
\sigma_t^2 &=& t^{-1} \tilde{
\E}\EEo_0\bigl[\bigl(\xi\cdot Y(t)\bigr)^2\bigr],
\\
\label{defsigma}
\sigma^2 &=& 2 \xi\cdot\Ahd\xi= \frac{2 \xi\cdot\Ah\xi}{\E[p]}.
\end{eqnarray}
It follows from results of~\cite{Kipnis-Varadhan-86} (or \cite
{DeMasi-Ferrari-89}, Theorem 2.1) that $\sigma_t^2$ tends to $\sigma^2$
as $t$ tends to infinity.
We now describe a Monte Carlo method to approximate~$\sigma_t^2$.
Using the definition of the tilted measure (\ref{deftdP}), one can see that
%
%
\begin{equation}
\label{enlevetilt} \sigma_t^2 = \frac{\tilde{\E}\EEo_0[(\xi\cdot
Y(t))^2]}{t} =
\frac
{\E\EEo_0[p(\omega) (\xi\cdot Y(t))^2]}{t \E[p]}.
\end{equation}
Assuming that we have easier access to the measure $\P$ than to the
tilted~$\tilde{\P}$, we prefer to base our Monte Carlo procedure on the
right-hand side of the second
identity in (\ref{enlevetilt}).
Let $Y^{(1)}, Y^{(2)},\ldots$ be independent random walks evolving in
the environments $\omega^{(1)}, \om^{(2)},\ldots\,$, respectively.
We write $\PP^{\overline{\om}}_0$ for their joint distribution, all random
walks starting from $0$, where $\overline{\om}$ stands for $(\om
^{(1)},\om
^{(2)},\ldots)$. The family of environments $\overline{\om}$ is itself
random, and we let $\P^\otimes$ be the product distribution with
marginal $\P$. In other words, under $\P^\otimes$, the environments
$\om^{(1)}, \om^{(2)},\ldots$ are independent and distributed
according to~$\P$. Our computable approximation of $\sigma_t^2$ is
defined by
%
%
\begin{equation}
\label{defhatA} \hat{A}_n(t) = \frac{p(\omega^{(1)}) (\xi\cdot
Y^{(1)}(t))^2 +
\cdots+ p(\omega^{(n)}) (\xi\cdot Y^{(n)}(t))^2}{nt \E[p]}.
\end{equation}
In $\hat{A}_n(t)$, the expectation $\E[p] = 2d\E[\omega_e]$ comes
into play. This expectation can be easily computed, so we assumed that
we did so beforehand.

The main result of this paper is the following optimal bounds on the
distribution of the error $|\hat{A}_n(t)- \sigma^2|$.
%
%
\begin{theorem}\label{thcomplete}
Under the assumptions \textup{(H1)--(H3)}, there exist $C,c > 0$ such
that, for any $n \in\N^*$, any $\eps> 0$ and any $t$ large enough,
%
%
\begin{equation}
\label{eqcomplete} \P^\otimes\PP^{\overline{\om}}_0 \bigl[ \bigl|
\hat{A}_n(t) - \sigma^2 \bigr| \ge\bigl(C\mu_d(t)+
\eps\bigr) /t \bigr] \le\exp\biggl( - \frac{n
\eps^2}{c t^2} \biggr),
\end{equation}
where $\sigma^2$ and $\hat{A}_n(t)$ are defined, respectively, in
(\ref{defsigma}) and (\ref{defhatA}), and
\[
\mu_d(t)=\cases{ \ln^q t, &\quad for $d = 2$,
\cr
1, &\quad for $d > 2$,}
\]
for some $q>0$ depending only on $\alpha$ and $\beta$.
\end{theorem}
This result precisely quantifies the convergence rate of a method
proposed by Papanicolaou in~\cite{Papanicolaou-83} in the beginning of
the eighties
to approximate the homogenized coefficients $\Ah$ numerically.

For completeness of the analysis we also prove a central limit theorem
(and identify the limiting variance) for the quantity $\sqrt{n(t)}
(\hat A_{n(t)}(t)-\sigma_t^2)$
for all $n\dvtx\N\to\N$ such that $n(t)$ tends to infinity with $t$.

Let us quickly discuss the sharpness of these results.
If $A$ was a periodic matrix (or even a constant matrix) we would get
the same estimate as in Theorem~\ref{thcomplete}, except in dimension
2 for which no logarithmic correction would be needed
[in the setting of Theorem~\ref{thcomplete}, we conjecture that $q=1$
is the optimal exponent in~(\ref{eqcomplete})].
Numerical tests illustrating (\ref{eqcomplete}) for $d=2$ are
reported and commented on in Section~\ref{sec6} of this article.

\subsection{Structure of the proof}\label{sec1.2} Although the result of
Theorem~\ref{thcomplete} is purely probabilistic (we estimate a distribution)
its proof involves both nontrivial probabilistic arguments (martingale
decomposition and Kipnis--Varadhan theory, large deviation estimates) and
nontrivial arguments of elliptic theory (Harnack inequality, De
Giorgi--Nash--Moser theory and $L^p$-theory).
What allows us to combine these arguments is spectral theory.
This makes the overall structure of the proof interesting and rather unusual.

The starting point of the proof is the observation that
\[
\bigl|\hat{A}_n(t) - \sigma^2 \bigr| \le\bigl|\hat{A}_n(t)-
\sigma_t^2 \bigr| + \bigl|\sigma_t^2-
\sigma^2 \bigr|.
\]
The result then follows from the following two estimates:
%
%
\begin{eqnarray}
\label{eqcomplete-syst}
\bigl|\sigma_t^2- \sigma^2 \bigr|&\leq& C
\frac{\mu_d(t)}{t},
\\
\label{eqcomplete-rand}
\P^\otimes\PP^{\overline{\om}}_0 \bigl[ \bigl|\hat{A}_n(t)-
\sigma_t^2 \bigr| \ge\eps/t \bigr]&\leq& \exp\biggl( -
\frac{n \eps^2}{c t^2} \biggr)
\end{eqnarray}
(see Theorems~\ref{syserr} and~\ref{randfluc}).
The second estimate is a large deviation estimate. Its proof is
standard once we are given sharp upper bounds on the transition
probabilities of the random walk in the random environment---which are
also by now standard under assumption (H3). The proof is given in
Section~\ref{secrandfluc} for completeness.
The central limit theorem for the quantity $\sqrt{n(t)} (\hat
A_{n(t)}(t)-\sigma_t^2)$ is given in Proposition~\ref{pclt} and
proved in Section~\ref{secclt}.

The core of this article is the estimate (\ref{eqcomplete-syst}). We
call its left-hand side the systematic error.
As proved in the celebrated paper~\cite{Kipnis-Varadhan-86} by Kipnis
and Varadhan
(see also~\cite{DeMasi-Ferrari-89}), the systematic error vanishes as
$t$ goes to infinity as soon as
the measure $\P$ is ergodic under translations.
The strategy to prove this result is to find a decomposition of
$Y(t)\cdot\xi$ into a martingale plus a remainder, in such a way that
the remainder term becomes negligible in the limit, and conclude using
the orthogonality of the increments of the martingale and ergodicity.
The approach taken up by~\cite{Kipnis-Varadhan-86} is based on the
spectral analysis of the (self-adjoint) operator of the environment
viewed by the particle. More precisely, it is shown that in order for
this decomposition with negligible remainder to exist, it suffices that
the spectral measure of this operator, once projected on the ``local
drift'' $\mfk{d}$ [see (\ref{deflocaldrift})], satisfies some
integrability condition (IC) at the edge of the spectrum. Condition
(IC) is then seen to be equivalent to asking $\mfk{d}$ to belong to
the function space $H^{-1}$, a fact which is automatically true due to
certain symmetry considerations that were systematized in \cite
{DeMasi-Ferrari-89}.

Our proof of (\ref{eqcomplete-syst}) consists of two steps.
We first make the argument of Kipnis and Varadhan quantitative in
Section~\ref{secquantKV}. That is, we show that stronger
integrability conditions than (IC) on the spectral measure can be
turned into quantitative estimates on the systematic error---this is a
general result
of independent interest.

In the second step, addressed in Section~\ref{secsyserr}, we prove
that indeed (IC) can be strengthened to higher integrability
properties, provided ergodicity is replaced by the stronger assumption
that the conductances are i.i.d., the hypothesis (H2).
This result is the main achievement of this article.
In~\cite{Gloria-Mourrat-10}, we had taken advantage of spectral theory
to turn results of~\cite{Gloria-Otto-09b} into bounds on spectral exponents.
In the present paper we go the other way around, and make systematic
use of the interplay between estimates on the spectral measure and
iterates of the
elliptic operator.
There is a twist in the analysis at this point.
In~\cite{Gloria-Mourrat-10} spectral theory is somehow only used at
the end of the argument
to re-phrase in terms of spectral exponents the results on systematic
errors obtained by PDE arguments in~\cite{Gloria-Otto-09b}.
Here spectral theory enters the proof itself and is used in combination
with PDE arguments.
This approach has the advantage of revealing the very nice structure of
the problem under consideration.

Let us point out that although the results of this paper are proved
under assumptions (H1)--(H3), the assumption (H2) on the statistics of
$\omega$
is only used to obtain the variance estimate of~\cite{Gloria-Otto-09},
Lemma 2.3. In particular, (H2) can be weakened as follows:
\begin{itemize}
\item the distribution of $\omega_{z,z+\mathbf{e}_i}$ may, in
addition, depend on $\mathbf{e}_i$,
\item independence can be replaced by finite correlation length
$C_L>0$, that is, for all $e,e'\in\mathbb{B}$, $\omega_e$ and
$\omega_{e'}$ are
independent if $|e-e'|\geq C_L$.
\end{itemize}

\textit{Notation.} So far we have already introduced the
probability measures $\PPo_0$ (distribution of $Y$), $\PP^{\overline{\om
}}_0$ (distribution of $Y^{(1)},Y^{(2)},\ldots$), $\P$
[i.i.d. distribution for $\omega= (\omega_e)_{e \in\mathbb{B}}$],
$\tilde{\P}$ [tilted measure defined in (\ref{deftdP})] and $\P^\otimes$
(product distribution of $\overline{\om}$ with marginal $\P$).
It will be convenient to define $\tilde{\P}^\otimes$ as the product
distribution of $\overline{\om}$ with marginal $\tilde{\P}$. For
convenience, we write $\P_0$ as a shorthand notation for $\P\PPo_0$,
$\tilde{\P}_0$ for $\tilde{\P}\PPo_0$, $\P^\otimes_0$ for $\P^\otimes
\PP^{\overline{\om}}_0$ and $\tilde{\P}^\otimes_0$ for $\tilde{\P
}^\otimes
\PP^{\overline{\om}}_0$. The corresponding expectations are written
accordingly, replacing ``P'' by ``E'' with the appropriate typography.
We write $|\cdot|$ for the Euclidian norm of $\R^d$.

Finally, $\lesssim$ and $\gtrsim$ stand, respectively, for $\le$ and
$\ge$ up to multiplicative constants (which depend only on the bounds
$\alpha$ and $\beta$ on the conductances
and the dimension~$d$, if not otherwise stated).

%
%
\section{Quantitative version of the Kipnis--Varadhan theorem}\label{sec2}
\label{secquantKV}

The Kipnis--Varadhan theorem~\cite{Kipnis-Varadhan-86} concerns
additive functionals of reversible Markov processes. It gives
conditions for such additive functionals to satisfy an invariance
principle. The proof of the result relies on a decomposition of the
additive functional as the sum of a martingale term plus a remainder
term, the latter being shown to be negligible. In this section, which
can be read independently of the rest of the paper, we give conditions
that enable us to obtain some quantitative bounds on this remainder term.

We consider discrete and continuous times simultaneously. Let $(\eta
_t)_{t \ge0}$ be a Markov process defined on some measurable state
space $\aleph$ (here, $t\ge0$ stands either for $t \in\N$ or for $t
\in\R_+$). We denote by $P_x$ the distribution of the process started
from $x \in\aleph$, and by $E_x$ the associated expectation. We
assume that this Markov process is reversible and ergodic with respect
to some probability measure~$\nu$. We write $P_\nu$ for the law of
the process started from the distribution $\nu$, and $E_\nu$ for the
associated expectation.

To the Markov process is naturally associated a semi-group $(P_t)_{t
\ge0}$ defined, for any $f \in\LL_2(\nu)$, by
\[
P_t f (x) = E_x\bigl[f(\eta_t)\bigr].
\]
Each $P_t$ is a self-adjoint contraction of $\LL^2(\nu)$. In the
continuous-time case, we assume further that the semi-group is strongly
continuous, that is to say, that $P_t f$ converges to $f$ in $\LL^2(\nu
)$ as $t$ tends to $0$, for any $f \in\LL^2(\nu)$. We let $L$
be the $\LL^2(\nu)$-infinitesimal generator of the semi-group. It is
self-adjoint in $\LL^2(\nu)$, and we fix the sign convention so that
it is a positive operator (i.e., $P_t = e^{-t L}$).\vadjust{\goodbreak}

Note that, in general, one can see using spectral analysis that there
exists a projection $\overline{P}$ such that $P_t f$ converges to
$\overline{P}
f$ as $t$ tends to $0$, $t > 0$. Changing $\LL^2(\nu)$ to the image
of the projection $\overline{P}$, and $P_0$ for $\overline{P}$, one
recovers a
strongly continuous semigroup of contractions, and one can still carry
the analysis below replacing $\LL^2(\nu)$ by the image of $\overline{P}$
when necessary.

In discrete time, we set $L = \mathrm{Id} - P_1$. Again, $L$ is a
positive self-adjoint operator on $\LL^2(\nu)$. Note that we slightly
depart from the custom of defining the generator as $P_1$ in order to
match more closely the continuous-time situation.

We denote by $\langle\cdot, \cdot\rangle$ the scalar product in
$\LL^2(\nu)$. For any function $f \in\LL^2(\nu)$ we define the
\textit{spectral measure} of $L$ projected on the function $f$ as the
measure $e_f$ on $\R_+$ that satisfies, for any bounded continuous
$\Psi\dvtx\R_+ \to\R$, the relation
%
%
\begin{equation}
\label{defef} \bigl\langle f, \Psi(L) f \bigr\rangle= \int\Psi(\lambda)
\,\d
e_f(\lambda).
\end{equation}

The Dirichlet form associated to $L$ is given by
%
%
\begin{equation}
\label{defnorm1} \|f\|_1^2 = \int\lambda\,\d
e_f(\lambda).
\end{equation}
We denote by $H^1$ the completion of the space $\{ f \in\LL^2(\nu)\dvtx
\|f\|_1 < + \infty\}$ with respect to\vspace*{1pt} this $\| \cdot\|_1$ norm, taken
modulo functions of zero $\| \cdot\|_1$ norm. This turns $(H^1,\|
\cdot\|_1)$ into a Hilbert space, and we let $H^{-1}$ denote its dual.
One can identify $H^{-1}$ with the completion of the space $\{ f \in
\LL^2(\nu)\dvtx\|f\|_{-1} < + \infty\}$ with respect to the norm $\|
\cdot\|_{-1}$ defined by
\[
\| f \|_{-1}^2 = \int\lambda^{-1} \,\d
e_f(\lambda).
\]
Indeed, for all $f \in\LL^2(\nu)$, the linear form
\[
\cases{\bigl(\LL^2(\nu) \cap H^1, \| \cdot
\|_1\bigr) \to\R,
\vspace*{1pt}\cr
\phi\mapsto\langle f, \phi\rangle,}
\]
has norm $\|f\|_{-1}$, and thus defines an element of $H^{-1}$ (with
norm $\|f\|_{-1}$) iff $\|f\|_{-1}$ is finite. The notion of spectral
measure introduced in (\ref{defef}) for functions of $\LL^2(\nu)$
can be extended to elements of $H^{-1}$. Indeed, let $\Psi\dvtx\R_+ \to
\R$ be a continuous function such that $\Psi(\lambda) = O(\lambda
^{-1})$ as $\lambda\to+\infty$. One can check that the map
\[
\cases{\bigl(\LL^2(\nu) \cap H^{-1}, \| \cdot
\|_{-1}\bigr) \to H^1,
\vspace*{1pt}\cr
f \mapsto\Psi(L) f,}
\]
extends to a bounded linear map on $H^{-1}$. One can then define the
spectral measure of $L$ projected on the function $f$ as the measure
$e_f$ such that for any continuous $\Psi$ with $\Psi(\lambda) =
O(\lambda^{-1})$, (\ref{defef}) holds. With a slight abuse of
notation, for all $f \in H^{-1}$ and $g \in H^1$, we write $\langle f,
g \rangle$ for the $H^{-1}-H^1$ duality product between $f$ and $g$.

For any $f \in H^{-1}$, we define $(Z_f(t))_{t \ge0}$ as
%
%
\begin{equation}
\label{defZf} Z_f(t) = \int_0^t
f(\eta_s) \,\d s \quad\mbox{or}\quad Z_f(t) = \sum
_{s = 0}^{t-1} f(\eta_s),
\end{equation}
according to whether we consider the continuous or the discrete time
cases. In the continuous case, the meaning of (\ref{defZf}) is unclear
a priori. Yet
it is proved in~\cite{DeMasi-Ferrari-89}, Lemma 2.4, that for any $t
\ge0$ the map
\[
\cases{\LL^2(\nu) \cap H^{-1} \to\LL^2(P_\nu),
\cr
f \mapsto Z_f(t),}
\]
can be extended by continuity to a bounded linear map on $H^{-1}$, and
moreover, that (\ref{defZf}) coincides with the usual integral as soon
as $f \in\LL^1(\nu)$. The following theorem is due to \cite
{DeMasi-Ferrari-89}, building on previous work of~\cite{Kipnis-Varadhan-86}.
%
%
\begin{theorem}
\label{kv}
\textup{(i)} For all $f \in H^{-1}$, there exists $(M_t)_{t \ge0}$, $(\xi
_t)_{t \ge0}$ such that $Z_f(t)$ defined in (\ref{defZf}) satisfies
the identity $Z_f(t) = M_t + \xi_t$, where $(M_t)$ is a
square-integrable martingale with stationary increments under $P_\nu$
(and the natural filtration), and $(\xi_t)$ is such that
%
%
\begin{equation}
\label{exitendvers0} t^{-1} E_\nu\bigl[(
\xi_t)^2\bigr] \xrightarrow{t \to+ \infty} 0.
\end{equation}
As a consequence, $t^{-1/2} Z_f(t)$ converges in law under $P_\nu$ to
a Gaussian random variable of variance $\sigma^2(f) < + \infty$ as
$t$ goes to infinity, and
%
%
\begin{equation}
\label{el2conv} t^{-1} E_\nu\bigl[\bigl(Z_f(t)
\bigr)^2\bigr] \xrightarrow{t \to+ \infty} \sigma^2(f).
\end{equation}

\textup{(ii)} If, moreover, $f \in\LL^1(\nu)$ and, for some $\overline{t} >
0$, $\sup_{0 \le t \le\overline{t}} |Z_f(t)|$ is in $\LL^2(\nu)$, then
the process $t \mapsto\sqrt{\eps} Z_f(\eps^{-1} t)$ converges in
law under $P_\nu$ to a Brownian motion of variance $\sigma^2(f)$ as
$\eps$ goes to $0$.
\end{theorem}

\textit{Remarks.}
The additional conditions appearing in statement (ii) are automatically
satisfied in discrete time, due to the fact that $H^{-1} \subset\LL
^2(\nu)$ in this case. In the continuous-time setting and when $f \in
\LL^1(\nu)$, the process $t \mapsto Z_f(t)$ is almost surely
continuous, and $\sup_{0 \le t \le\overline{t}} |Z_f(t)|$ is indeed a
well-defined random variable.

Under some additional information on the spectral measure of $f$, we
can estimate the rates of convergence in the limits (\ref
{exitendvers0}) and (\ref{el2conv}). For any $\gamma> 1$ and $q \ge
0$, we say that the spectral exponents of a function $f \in H^{-1}$ are
at least $(\gamma,-q)$ if
%
%
\begin{equation}
\label{specexp} \int_0^\mu\d e_f(
\lambda) = O \bigl( \mu^{\gamma} \ln^{q}\bigl(\mu^{-1}
\bigr) \bigr) \qquad(\mu\to0).
\end{equation}
Note that the phrasing is consistent, since if $(\gamma',-q') \le
(\gamma,-q)$ for the lexicographical order, and if the spectral
exponents of $f$ are at least $(\gamma,-q)$, then they are at least
$(\gamma',-q')$. In~\cite{Mourrat-10}, it was found more convenient
to consider, instead of~(\ref{specexp}), a condition of the following form:
%
%
\begin{equation}
\label{specexp2} \int_0^\mu
\lambda^{-1} \,\d e_f(\lambda) = O \bigl(
\mu^{\gamma-1} \ln^{q}\bigl(\mu^{-1}\bigr) \bigr) \qquad(\mu
\to0).
\end{equation}
One can easily check that conditions (\ref{specexp}) and (\ref
{specexp2}) are equivalent. Indeed, on the one hand, one has the
obvious inequality
\[
\int_0^\mu\d e_f(\lambda) \le\mu
\int_0^\mu\lambda^{-1} \,\d
e_f(\lambda),
\]
which shows that (\ref{specexp2}) implies (\ref{specexp}). On the
other hand, one may perform a kind of integration by parts (use
Fubini's theorem),
\begin{eqnarray*}
\int_0^\mu\lambda^{-1} \,\d
e_f(\lambda) & = & \int_0^\mu\int
_{\delta= \lambda}^{+\infty} \delta^{-2} \,\d\delta\,\d
e_f(\lambda)
\\
& = & \int_{\delta= 0}^{+\infty} \delta^{-2} \int
_{\lambda=
0}^{\delta\wedge\mu} \d e_f(\lambda) \,\d\delta
\end{eqnarray*}
and obtain the converse implication by examining separately the
integration over $\delta$ in $[0,\mu)$ and in $[\mu,+\infty)$.

For all $\gamma> 1$ and $q \ge0$, we set
%
%
\begin{equation}
\label{defpsi} \psi_{\gamma,q}(t) = \cases{t^{1-\gamma}
\ln^q(t), &\quad if $\gamma< 2$,
\vspace*{1pt}\cr
t^{-1} \ln^{q+1}(t),
&\quad if $\gamma= 2$,
\vspace*{1pt}\cr
t^{-1}, &\quad if $\gamma> 2$.}
\end{equation}
The quantitative version of Theorem~\ref{kv} is as follows.

%
%
\begin{theorem}
\label{quantkv}
If the spectral exponents of $f \in H^{-1}$ are at least $(\gamma,-q)$,
then the decomposition $Z_f(t) = M_t + \xi_t$ of Theorem~\ref{kv}
holds with the additional property that
\[
t^{-1} E_\nu\bigl[(\xi_t)^2\bigr]
= O\bigl(\psi_{\gamma,q}(t)\bigr) \qquad(t \to+ \infty).
\]
Moreover,
\[
\sigma^2(f) - \frac{E_\nu[Z_f(t)^2]}{t} = O\bigl(\psi_{\gamma,q}(t)
\bigr) \qquad(t \to+ \infty).
\]
\end{theorem}
\begin{pf}
In the continuous-time setting, the argument for the first estimate is
very similar to the one of~\cite{Mourrat-10}, Proposition 8.2, and we
do not repeat the details here.
It is based on the observation that
%
%
\begin{equation}
\label{normxi2} \frac{1}{t} E_\nu\bigl[(\xi_t)^2
\bigr] = 2 \int\frac{1-e^{-\lambda
t}}{\lambda^2 t} \,\d e_f(\lambda).
\end{equation}
One needs to take into account the possible logarithmic terms that
appear in (\ref{specexp2}) and which are not considered in~\cite{Mourrat-10}.
Some care is also needed because we do not assume that $f \in\LL^2(\nu
)$. Yet one can easily replace the bound involving the $\LL^2(\nu)$
norm of $f$ by its $H^{-1}$ norm. The second part of the
statement is given by~\cite{Mourrat-10}, Proposition 8.3.

We now turn to the discrete time setting. In this context, identity
(\ref{normxi2}) should be replaced by
\[
\frac{1}{t} E_\nu\bigl[(\xi_t)^2
\bigr] = 2 \int\frac{1-(1-\lambda
)^t}{\lambda^2 t} \,\d e_f(\lambda).
\]
By definition, $L=\mathrm{Id} - P_1$, where $P_1$ is the semi-group at
time $1$. Hence, the spectrum of $L$ is contained in $[0,2]$.
One can then follow the same computations as before to prove the first
part of Theorem~\ref{quantkv}.

Somewhat surprisingly, the second part of the statement requires
additional attention in the discrete time setting. Indeed, in the
continuous case, the argument of~\cite{Mourrat-10}, Proposition 8.3
(which already appears in~\cite{DeMasi-Ferrari-89}) is that $Z_f(t)$
and $\xi(t)$ are orthogonal in $\LL^2(P_\nu)$, a fact obtained using
the invariance under time symmetry. This orthogonality is only
approximately valid in the discrete-time setting. Indeed, let us recall
that $Z_f(t)$ is given by (\ref{defZf}), while $\xi_t$ is obtained as
the limit in $\LL^2(P_\nu)$ of
\[
- u_\eps(\eta_t) + u_\eps(
\eta_0),
\]
where $u_\eps= (\eps+ L)^{-1} f$. Using time symmetry, what we obtain
is that $\xi_t$ is orthogonal to $(Z_f(t) + f(\eta_t))$.
As a consequence, the cross-product
$
E_\nu[Z_f(t) \xi_t]
$, which is equal to $0$ in the proof of~\cite{Mourrat-10},
Proposition 8.3, is in the present case equal to $- E_\nu[f(\eta_t)
\xi_t]$.
Yet spectral analysis ensures that this term is equal to
\[
\int\frac{1 - (1-\lambda)^t}{\lambda} \,\d e_f(\lambda) = O(1)
\qquad(t \to+\infty),
\]
which is what we need to obtain the second claim of the theorem.
\end{pf}

%
%
\section{The systematic error}\label{sec3}
\label{secsyserr}

We now come back to the analysis of the Monte Carlo approximation of
the homogenized coefficients within assumptions (H1)--(H3).
The aim of this section is to estimate the difference between $\sigma
_t^2$ and the quantity $\sigma^2$ we wish to approximate [both being
defined in (\ref{defsigmat})].
This difference, that we
refer to as the \textit{systematic error} after~\cite{Gloria-Otto-09},
is shown to be of order $1/t$ as $t$ tends to infinity, up to a
logarithmic correction in dimension $2$.
%
%
\begin{theorem}
\label{syserr}
Under assumptions \textup{(H1)--(H3)}, there exists $q \ge0$ such that,
as $t$ tends to infinity,
%
%
\begin{equation}
\label{eqsyserr} \sigma_t^2 - \sigma^2 =
\cases{O \bigl( t^{-1} \ln^{q}(t) \bigr), &\quad if $d = 2$,
\vspace*{1pt}\cr
O
\bigl( t^{-1} \bigr), &\quad if $d >2$.}
\end{equation}
\end{theorem}
Theorem~\ref{syserr} is a discrete-time version of~\cite{Mourrat-10},
Corollary 2.6.
Its proof makes use of an auxiliary process that we now introduce.

Let $(\theta_x)_{x \in\Z^d}$ be the translation group that acts on
the set of environments as follows: for any pair of neighbors $y,z \in
\Z^d$, $(\theta_x \omega)_{y,z} = \omega_{x+y,x+z}$. The
\textit{environment viewed by the particle} is the process defined by
%
%
\begin{equation}
\label{defenvpart} \omega(t) = \theta_{Y(t)} \omega.
\end{equation}
One can check that $(\omega(t))_{t \in\N}$ is a Markov chain whose
generator is given by
%
%
\begin{equation}
\label{eqdef-L} - \L f (\omega) = \frac{1}{p(\omega)} \sum
_{|z| = 1} \omega_{0,z} \bigl(f(\theta_z
\omega) - f(\omega)\bigr),
\end{equation}
so that $\EEo_0[f(\omega(1))] = (I-\L)f(\omega)$. Moreover, the
measure $\tilde{\P}$ defined in (\ref{deftdP}) is reversible and
ergodic for this process~\cite{DeMasi-Ferrari-89}, Lemma 4.3(i). As a
consequence, the operator $\L$ is (positive and) self-adjoint in $\LL
^2(\tilde{\P})$.

The proof of Theorem~\ref{syserr} relies on spectral analysis. For any
function $f \in L^2(\tilde{\P})$, let $e_f$ be the spectral measure of
$\L$ projected on the function $f$. This measure is such that, for any
positive continuous function $\Psi\dvtx[0,+\infty) \to\R_+$, one has
\[
\tilde{\E}\bigl[f \Psi(\L)f\bigr] = \int\Psi(\lambda) \,\d e_f(
\lambda).
\]
For any $\gamma> 1$ and $q \ge0$, we recall that we say that the
spectral exponents of a function $f$ are at least $(\gamma,-q)$ if
(\ref{specexp}) holds.

Let us define the local drift $\mfk{d}$ in direction $\xi$ as
%
%
\begin{equation}
\label{deflocaldrift} \mfk{d}(\omega) = \EEo_0\bigl[\xi\cdot Y(1)
\bigr] = \frac{1}{p(\omega)} \sum_{|z| = 1}
\omega_{0,z} \xi\cdot z.
\end{equation}
As we shall prove at the end of this section, we have the following
bounds on the spectral exponents of $\mfk{d}$.
%
%
\begin{prop}
\label{pspecexp}
Under assumptions \textup{(H1)--(H3)}, there exists $q \ge0$ such that
the spectral exponents of the function $\mfk{d}$ are at least
%
%
\begin{equation}
\label{eqspecexp} \cases{(2,-q), &\quad if $d = 2$,
\vspace*{1pt}\cr
(d/2 + 1,0), &\quad if $3 \le d
\le5$,
\vspace*{1pt}\cr
(4,-1), &\quad if $d = 6$,
\vspace*{1pt}\cr
(4,0), &\quad if $d \ge7$.}
\end{equation}
\end{prop}
Let us see how this result implies Theorem~\ref{syserr}. In order to
do so, we also need the following information that is a consequence of
Proposition~\ref{pspecexp}.
%
%
\begin{cor}
\label{corpoly}
Let
%
%
\begin{equation}
\label{defdt} \mfk{d}_t(\omega) = \EEo_0\bigl[\mfk{d}
\bigl(\omega(t)\bigr)\bigr]\vadjust{\goodbreak}
\end{equation}
be the image of $\mfk{d}$ by the semigroup at time $t$ associated with
the Markov chain $(\omega(t))_{t \in\N}$. There exists $q \ge0$
such that
\[
\tilde{\E}\bigl[(\mfk{d}_t)^2\bigr] = \cases{ O
\bigl(t^{-2} \ln^q(t) \bigr), &\quad if $d = 2$,
\vspace*{2pt}\cr
O
\bigl(t^{-(d/2 + 1)} \bigr), &\quad if $3 \le d \le5$,
\vspace*{2pt}\cr
O \bigl(t^{-4}
\ln(t) \bigr), &\quad if $d = 6$,
\vspace*{2pt}\cr
O \bigl(t^{-4} \bigr), &\quad if $d \ge7$.}
\]
\end{cor}
\begin{pf}
This result is the discrete-time analog of~\cite{Gloria-Mourrat-10},
Corollary 1. It is obtained the same way, noting that
\[
\tilde{\E}\bigl[(\mfk{d}_t)^2\bigr] = \int(1-
\lambda)^{2t} \,\d e_{\mfk
{d}}(\lambda)
\]
and that the support of the measure $e_{\mfk{d}}$ is contained in $[0,2]$.
\end{pf}
We are now in position to prove Theorem~\ref{syserr}.
\begin{pf*}{Proof of Theorem~\ref{syserr}}
The proof has the same structure as for the continuous-time case of
\cite{Mourrat-10}, Proposition 8.4.
Note that~\cite{DeMasi-Ferrari-89}, Theorem 2.1, ensures that
%
%
\begin{equation}
\label{eqdemasi} \lim_{t \to\infty}\sigma_t^2
\stackrel{(\mathrm{def})} {=} \lim_{t \to\infty} t^{-1}{
\E}_0\bigl[\bigl(\xi\cdot Y(t)\bigr)^2\bigr] =
\sigma^2.
\end{equation}
The starting point is the observation that, under $\tilde{\P}_0$, the
process defined by
%
%
\begin{equation}
\label{defNt} N_t = \xi\cdot Y(t) - \sum
_{s = 0}^{t-1} \mfk{d}\bigl(\omega(s)\bigr)
\end{equation}
is a square-integrable martingale with stationary increments.
On the one hand, following (\ref{defZf}), we denote by $Z_\mfk{d}(t)$
the sum appearing in the right-hand side of (\ref{defNt}).
From Proposition~\ref{pspecexp} and Theorem~\ref{quantkv}, we learn
that there exist $\overline{\sigma}$ and $q \ge0$ such that
%
%
\begin{equation}
\label{controlZd} t \overline{\sigma}^2 - \tilde{\E}_0\bigl[
\bigl(Z_{\mfk{d}}(t)\bigr)^2\bigr] = \cases{ O \bigl(
\ln^{q}(t) \bigr), &\quad if $d = 2$,
\vspace*{1pt}\cr
O ( 1 ), &\quad if $d >2$.}
\end{equation}
On the other hand, since $N_t$ is a martingale with stationary increments,
%
%
\begin{equation}
\label{eqmartingale} \tilde{\E}_0\bigl[(N_t)^2
\bigr] = t \tilde{\E}_0\bigl[(N_1)^2\bigr].
\end{equation}
As in the proof of Theorem~\ref{quantkv} in the discrete time case, we
then use that $\xi\cdot Y(t)$ is orthogonal to $(Z_\mfk{d}(t) + \mfk
{d}(\omega(t)))$
to turn (\ref{defNt}) into
%
%
\begin{eqnarray}
\label{decompNt} t^{-1}\tilde{\E}_0\bigl[(N_t)^2
\bigr] &=& t^{-1}\tilde{\E}_0\bigl[\bigl(\xi\cdot Y(t)
\bigr)^2\bigr] + t^{-1}\tilde{\E}_0\bigl[
\bigl(Z_{\mfk{d}}(t)\bigr)^2\bigr]\nonumber\\[-8pt]\\[-8pt]
&&{} + 2t^{-1} \tilde{
\E}_0\bigl[\mfk{d}\bigl(\omega(t)\bigr) \bigl(\xi\cdot Y(t)\bigr)
\bigr].\nonumber
\end{eqnarray}
We already control the left-hand side and the second term of the
right-hand side of (\ref{decompNt}).
In order to quantify the convergence\vadjust{\goodbreak} of $t^{-1}\tilde{\E}_0[(\xi\cdot
Y(t))^2]$ it remains to control the last term.
In particular, provided we show that
%
%
\begin{equation}
\label{controlrest} \tilde{\E}_0\bigl[\mfk{d}\bigl(\omega(t)\bigr)
\bigl(\xi\cdot Y(t)\bigr)\bigr] = \cases{ O \bigl( \ln^{q}(t) \bigr),
&\quad
if $d = 2$,
\cr
O ( 1 ), &\quad if $d >2$,}
\end{equation}
(\ref{decompNt}), (\ref{controlZd}), (\ref{eqmartingale}) and
(\ref{eqdemasi}) imply first that $\sigma^2=\tilde{\E
}_0[(N_1)^2]-\overline
{\sigma}{}^2$,
and then the desired quantitative estimate (\ref{eqsyserr}).
We now turn to (\ref{controlrest}) and write
\begin{eqnarray*}
\tilde{\E}_0\bigl[\mfk{d}\bigl(\omega(t)\bigr) \bigl(\xi\cdot Y(t)
\bigr)\bigr] & = & \sum_{s=0}^{t-1} \tilde{
\E}_0\bigl[\mfk{d}\bigl(\omega(t)\bigr) \bigl(\xi\cdot\bigl(Y(s+1)-Y(s)
\bigr)\bigr)\bigr]
\\
& = & \sum_{s=0}^{t-1} \tilde{
\E}_0\bigl[\mfk{d}_{t-s-1}\bigl(\omega(s+1)\bigr) \bigl(\xi\cdot
\bigl(Y(s+1)-Y(s)\bigr)\bigr)\bigr],
\end{eqnarray*}
where we have used the Markov property at time $s+1$, together with the
definition (\ref{defdt}) of $\mfk{d}_{t-s-1}$.
Using Cauchy--Schwarz\vspace*{1pt} inequality and the stationarity of the process
$(\omega(t))_{t \in\N}$ under $\tilde{\E}_0$, this sum is bounded by
\[
|\xi|^2 \sum_{s=0}^{t-1} {\tilde{
\E}\bigl[(\mfk{d}_{t-s-1})^2\bigr]}^{1/2}.
\]
Estimate (\ref{controlrest}) then follows from Corollary \ref
{corpoly}. The proof of the theorem is complete.
\end{pf*}
Proposition~\ref{pspecexp} is a discrete-time counterpart of \cite
{Gloria-Mourrat-10}, Theorem 5.
In~\cite{Gloria-Mourrat-10}, Theorem 5, however, we had proved, in
addition, that the spectral exponents are at least $(d/2-2,0)$,
which is sharper than the exponents of Proposition~\ref{pspecexp} for $d>10$.
In particular, for $d>10$ the bounds of~\cite{Gloria-Mourrat-10},
Theorem 5, follow from
results of~\cite{Mourrat-10}, whose adaptation to the discrete time
setting is not straightforward.
As shown above, the present statement is sufficient to prove the
optimal scaling of the systematic error, and
we do not investigate further this issue (see, however, Remark \ref
{remopt-spex}).
The proof of Proposition~\ref{pspecexp} is rather involved and one
may wonder whether this is worth the effort in terms of the
application we have in mind, namely, Theorem~\ref{syserr}.
In order to obtain the optimal convergence rate in Theorem~\ref
{syserr} we need
the spectral exponents to be larger than $(2,0)$. Proving that the
exponents are at least $(2,0)$ is rather direct using
results of~\cite{Gloria-Otto-09} (see the first three steps of the
proof of Proposition~\ref{pspecexp}). Yet proving that they are
larger than $(2,0)$ for $d>2$ is as involved as proving
Proposition~\ref{pspecexp} itself.
This is the reason why we display the complete proof of
Proposition~\ref{pspecexp}---although the precise values of the
spectral exponents are not that important
in the context of this paper.\looseness=-1

There are two new features in the proof of Proposition~\ref{pspecexp}
with respect to our previous works:
\begin{itemize}
\item First, the discrete elliptic operator we consider here is
slightly different than the operator considered\vadjust{\goodbreak} in \cite
{Gloria-Otto-09} since
the zero-order term is now random as well---the adaptation of the
results of~\cite{Gloria-Otto-09} is only technical though;
\item The string of arguments is different than in the proof of \cite
{Gloria-Mourrat-10}, Theorem 5.
In particular, the starting point of~\cite{Gloria-Mourrat-10} was an
estimate obtained in~\cite{Gloria-Otto-09b}
based on the crucial use of a covariance estimate.
In~\cite{Gloria-Otto-09b} the main quantity of interest was a
systematic error.
In the present proof the main quantity of interest is the spectral
exponents at the first place.
This twist of points of view allows us to reduce the proof to a
suitable use of the variance estimate
only, and reveals the general structure of the problem.
\end{itemize}
This proof does not only complete the proof of Theorem~\ref{syserr}
but allows us to shed
some new light on our conjecture in~\cite{Gloria-Mourrat-10} on the
optimal values of the spectral exponents (see Remark~\ref{remopt-spex}).

As already mentioned, this proof makes extensive use of tools developed
by the authors, and by Otto.
For the reader's convenience, we recall five useful auxiliary results
from~\cite{Gloria-Otto-09,Gloria-Otto-09b,Gloria-10}: a spectral gap estimate
and bounds on Green's functions.
%
%
\begin{lem}[(\cite{Gloria-Otto-09}, Lemma 2.3)]\label{lemvar-estim}
Let $a=\{a_i\}_{i\in\mathbb{N}}$ be a sequence of i.i.d. random
variables with range $[\alpha,\beta]$.
Let $X$ be a Borel measurable function of $a\in\mathbb{R}^\mathbb
{N}$ (i.e., measurable
w.r.t. the smallest $\sigma$-algebra on $\mathbb{R}^\mathbb{N}$
for which all coordinate functions $\mathbb{R}^\mathbb{N}\ni a\mapsto
a_i\in\mathbb{R}$
are Borel measurable; cf.~\cite{Klenke-08}, Definition~14.4).

Then we have
%
%
\begin{equation}
\label{eqvar-estim} \operatorname{var} [X ] \le\Biggl\langle\sum
_{i=1}^\infty\sup_{a_i} \biggl\llvert
\frac{\partial X}{\partial a_i}\biggr\rrvert^2 \Biggr\rangle
\operatorname{var}
[a_1 ],
\end{equation}
where ${\sup_{a_i}}\llvert\frac{\partial X}{\partial a_i}\rrvert$ denotes
the supremum of the modulus of
the $i$th partial derivative
\[
\frac{\partial X}{\partial a_i}(a_1,\ldots,a_{i-1},a_i,a_{i+1}
,\ldots)
\]
of $X$
with respect to the variable $a_i\in[\alpha,\beta]$.
\end{lem}
Let $h\dvtx\Z^d\to\R$ be some function. We define its forward and
backward discrete gradients $\nabla$ and $\nabla^*$
as
\[
\nabla h(x):=\lleft[ \matrix{ h(x+\ee_1)-h(x)
\cr
\vdots
\cr
h(x+\ee_d)-h(x) } \rright],\qquad \nabla^* h(x):=\lleft[
\matrix{ h(x)-h(x-\ee_1)
\cr
\vdots
\cr
h(x)-h(x-\ee_d) }
\rright];
\]
the discrete backward divergence of some vector field $V\dvtx\Z^d\to\R^d$
is given by the ``formal'' scalar product between $\nabla^*$ and
$V$, that is,
\[
\nabla^*\cdot V(x) = \sum_{i=1}^d \bigl(
V_i(x+\ee_i)-V_i(x) \bigr).
\]
To avoid confusion, when a function $h\dvtx\Z^d\times\Z^d\to\R
,(x,z)\mapsto h(x,z)$ depends on two variables, we denote
by $\nabla_1h$ (resp., $\nabla_1^*h$) the forward (resp., backward)
discrete gradient with respect to the first variable ($x$ here) and by
$\nabla_2h$ (resp., $\nabla_2^*h$) the forward (resp., backward)
discrete gradient with respect to the second variable ($z$ here). We
further use the notation $\nabla_{k,i}h:=\nabla_k h \cdot\ee_i$ for
the forward discrete gradients in direction $\ee_i$ (and likewise for
the backward gradients), $i\in\{1,\ldots,d\}$.

We define discrete Green's functions as follows.
%
%
\begin{defi}[(Discrete Green's function)]\label{defGreen}
Let $d\geq2$. Let $\omega$ be an environment, $p_\omega\dvtx\Z^d\to\R
, x\mapsto\sum_{|z-x|=1}\omega_{x,z}$ and $A$ be the associated
diagonal matrix on $\Z^d$ defined
by $A(x)=\operatorname{diag}(\omega_{x,x+\ee_1},\ldots,\omega_{x,x+\ee_d})$.
For all $T>0$, the Green function $G_T(\cdot,\cdot;\omega)\dvtx\Z
^d\times\Z
^d\to\Z^d,(x,y)\mapsto G_T(x,y;\omega)$ associated with
the environment $\omega$
is defined for all $y\in\Z^d$ as the unique space square-integrable
solution to
%
%
\begin{eqnarray}
\label{eqdisc-Green}
&&
\int_{\Z^d}T^{-1}p_\omega(x)G_T(x,y;
\omega)v(x) \,\d x\nonumber\\[-8pt]\\[-8pt]
&&\qquad{} +\int_{\Z
^d}\nabla v(x)\cdot A(x)
\nabla_1 G_T(x,y;\omega) \,\d x=v(y)\nonumber
\end{eqnarray}
for all square-integrable functions $v\dvtx\Z^d\to\R$, where $\int_{\Z
^d} \d y$ denotes the sum over all $y \in\Z^d$.
\end{defi}

The existence and uniqueness of discrete Green's functions is a
consequence of Riesz's representation theorem.
In the rest of this article we use the shorthand notation $G_T(x,y)$
for $G_T(x,y;\omega)$.
Note that $G_T$ is stationary in the sense that $(x,y)\mapsto
G_T(x+z,y+z)$ has the same statistics as $(x,y)\mapsto G_T(x,y)$.
This will be used for the gradient of the Green function as follows:
for all $q>0$,
%
%
\begin{equation}
\label{eqstat-grad-Green} \bigl\langle\bigl|\nabla_2 G_T(x,y)\bigr|^q
\bigr\rangle= \bigl\langle\bigl|\nabla_1 G_T(x-y,0)\bigr|^q
\bigr\rangle.
\end{equation}
The next two lemmas give estimates on the Green function and its derivatives.
%
%
\begin{lem}[(\cite{Gloria-10}, Lemma 3.2)]\label{lemptwise-Green-decay-opt}
There exists $c>0$ depending only on $\alpha,\beta$ and $d$ such that
for every environment $\omega$ and for all $T>0$,
the Green function $G_T$ satisfies the pointwise estimates: for all
$x,y\in\Z^d$,
%
%
\begin{eqnarray}
\label{eqGTvsgT-opt-d>2}
\mbox{for }d>2\qquad G_T(x,y) &\lesssim&
\bigl(1+|x-y|\bigr)^{2-d}\exp\biggl( -c\frac{|x-y|}{\sqrt{T}} \biggr),
\\
\label{eqGTvsgT-opt-d=2}
\mbox{for }d=2\qquad G_T(x,y) &\lesssim& \ln
\biggl(\frac{\sqrt{T}}{1+|x-y|}\biggr) \exp\biggl( -c\frac{|x-y|}{\sqrt
{T}} \biggr).
\end{eqnarray}
\end{lem}
%
%
\begin{lem}[(\cite{Gloria-Otto-09}, Lemma 2.9)]\label{lemint-grad}
Let $\omega$ be an environment, $T>0$, and let $G_T$ be the associated
Green function. Then, for $d\geq2$, there exists ${p}>2$ depending
only on $\alpha,\beta$ and $d$ such that for all $T>0$, ${p}\geq
r\geq2$,
$k >0$ and $R\lesssim1$,
%
%
\begin{equation}\label{eqint-grad}
\int_{R\leq|z|\leq2R}\bigl|\nabla_1G_T(z,0)
\bigr|^r\,\d z
\lesssim R^d\bigl(R^{1-d}
\bigr)^r \min\bigl\{1,\sqrt{T}R^{-1}\bigr\}^{k}.
\end{equation}
\end{lem}
Note that this lemma shows that $\nabla_1 G_T(z,0)$ has the optimal
decay $(1+|z|)^{1-d}$ (i.e., the decay of the Green function of the
Laplace operator)
when integrated on dyadic annuli (plus the exponential, or
superalgebraic decay).

%
\begin{cor}[(\cite{Gloria-Otto-09}, Corollary 2.3)]\label{coro}
For every environment $\omega$ and for all $T>0$ and $x,y\in\Z^d$,
\[
\bigl|\nabla_1 G_T(x,y;\omega)\bigr|,\bigl|\nabla_2
G_T(x,y;\omega)\bigr| \lesssim1
\]
(the multiplicative constant depending only on $\alpha,\beta$ and $d)$.
\end{cor}
Note that the versions of these lemmas proved in~\cite{Gloria-Otto-09}
and~\cite{Gloria-10} cover the case
when the zero-order term is constant, namely, $T^{-1}$ in place of
$T^{-1}p_\omega(x)$.
The proofs adapt mutatis mutandis using the uniform bounds $0<2^d\alpha
\leq p_\omega\leq2^d\beta$.

The last lemma we shall need is the following double convolution estimate.
%

\begin{lem}[(\cite{Gloria-Mourrat-10}, Lemma 6)]\label{lemdb-conv}
Let $d>2$, $T\gg1$ and let $g_T\dvtx\Z^d\to\R^+$ be given by
\[
g_T(x) = \bigl(1+|x|\bigr)^{2-d}\exp\biggl( -c\frac{|x|}{\sqrt{T}}
\biggr)
\]
for some $c>0$.
Let $h_T\dvtx\Z^d\to\R^+$ be such that
\[
\int_{|x|\leq R}h_T(x)^2 \lesssim1
\]
and for all $R\gg1$ and all $j\in\N$,
\[
\int_{2^jR \leq|x|<2^{j+1} R}h_T(x)^2\,\d x \lesssim
\bigl(2^jR\bigr)^{d-2(d-1)}.
\]
Then,
%
%
\begin{eqnarray} \label{eqap-8}
&&
\int_{\Z^d}\int_{\Z^d}\int
_{\Z^d}g_T(w)g_T\bigl(w'
\bigr)h_T(z-w)h_T\bigl(z-w'\bigr)\,\d z\,\d w
\,\d w'
\nonumber\\[-8pt]\\[-8pt]
&&\qquad\lesssim1+ \cases{ T^{3-d/2}, &\quad if $5\ge d>2$,
\vspace*{1pt}\cr
\ln T, &\quad if $d=6$,
\vspace*{1pt}\cr
1, &\quad if $d>6$.}
\nonumber
\end{eqnarray}
\end{lem}

We are in position to prove Proposition~\ref{pspecexp}.\vadjust{\goodbreak}
\begin{pf*}{Proof of Proposition~\ref{pspecexp}}
Our starting point is the following inequality which holds for every
nonnegative measure $\kappa$:
%
%
\begin{equation}
\label{eqstep-0-0}
\int_0^{T^{-1}} \d\kappa(
\lambda) \lesssim T^{-4}\int_0^\infty
\frac{1}{(T^{-1}+\lambda)^4} \,\d\kappa(\lambda),
\end{equation}
which follows from the fact that for $\lambda\leq T^{-1}$, $\frac
{T^{-4}}{(T^{-1}+\lambda)^4}\gtrsim1$. The variable $T^{-1}$
for $T$ large plays the role of $\mu$ in (\ref{specexp}).
In what follows we make the standard identification between stationary
functions $(z,\omega)\mapsto f(z,\omega)$ of both the space variable
$z\in\Z^d$
and the environment $\omega$ and their translated versions at 0
$\omega\mapsto f(0,\theta_z\omega)$ depending on the environment only.
We define $\phi_T$ as the unique stationary solution to
%
%
\begin{equation}
\label{eqphiT} T^{-1}\phi_T(x)-\frac{1}{p_\omega(x)}\nabla^*
\cdot A(x)\nabla\phi_T = \frac{1}{p_\omega(x)}\nabla^*\cdot A(x)\xi,
\end{equation}
whose existence and uniqueness follow from the Riesz representation
theorem in $\mathbb{L}^2(\tilde{\mathbb{P}})$
using the identification between the stationary function $\phi_T$ and
its version defined on the environment only
(see a similar argument of~\cite{Kunnemann-83}).
In particular, with the notation $\mathfrak{d}=\frac{1}{p_\omega
(x)}\nabla^*\cdot A(x)\xi$,
\[
\phi_T = \bigl(T^{-1}+\mathcal{L}\bigr)^{-1}
\mathfrak{d},
\]
where $\mathcal{L}$ is the operator defined in (\ref{eqdef-L}), and
the spectral theorem ensures
that
\[
\tilde{\mathbb{E}}\bigl(\phi_T^2\bigr) = \tilde{
\mathbb{E}}\bigl(\mathfrak{d}\bigl(T^{-1}+\mathcal{L}
\bigr)^{-2}\mathfrak{d}\bigr) = \int_0^\infty
\frac
{1}{(T^{-1}+\lambda)^2}\,\d e_{\mathfrak{d}}(\lambda),
\]
where $e_{\mathfrak{d}}$ is the spectral measure of $\mathcal{L}$
projected on the drift $\mathfrak{d}$.
We also let $\psi_T$ be the unique stationary solution to
%
%
\begin{equation}
\label{eqpsiT} T^{-1}\psi_T(x)-\frac{1}{p_\omega(x)}\nabla^*
\cdot A(x)\nabla\psi_T(x) = \phi_T(x),
\end{equation}
whose existence and uniqueness also follows from the Riesz
representation theorem in the probability space as well.
This time,
\[
\psi_T = \bigl(T^{-1}+\mathcal{L}\bigr)^{-2}
\mathfrak{d},
\]
and the spectral theorem yields
\[
\tilde{\mathbb{E}}\bigl(\psi_T^2\bigr) = \tilde{
\mathbb{E}}\bigl(\mathfrak{d}\bigl(T^{-1}+\mathcal{L}
\bigr)^{-4}\mathfrak{d}\bigr) = \int_0^\infty
\frac
{1}{(T^{-1}+\lambda)^4}\,\d e_{\mathfrak{d}}(\lambda).
\]
From now on, we shall use the shorthand notation $ \langle u
\rangle:=\tilde
{\mathbb{E}}(u)$
and $\operatorname{var} [u ]= \langle(u- \langle
u \rangle)^2 \rangle$ for all $u\in\LL^2(\tilde
{\mathbb{P}})$.
In particular, the identity above turns into
%
%
\begin{equation}
\label{eqstep-0-1} \int_0^\infty
\frac{1}{(T^{-1}+\lambda)^4} \,\d e_{\mathfrak
{d}}(\lambda) = \operatorname{var} [
\psi_T ],
\end{equation}
since $ \langle\psi_T \rangle=\frac{1}{\mathbb
{E}[p]}\int\psi_T p \,\d
\mathbb{P}=\frac{T}{\mathbb{E}[p]}\int\phi_T p \,\d\mathbb{P} =0$
using equations (\ref{eqpsiT}) and (\ref{eqphiT}).

The streamline of the proof is to obtain bounds on the spectral
exponents via (\ref{eqstep-0-0}) and (\ref{eqstep-0-1}) by proving
bounds on the variance of $\psi_T$.

The rest of the proof, which is dedicated to the estimate of
$\operatorname{var} [\psi_T ]$, is divided into five steps.
As a starting point we appeal to the variance estimate of Lem\-ma~\ref
{lemvar-estim} that we apply to $\psi_T$.
This requires us to estimate the susceptibility of $\psi_T$ with
respect to the random coefficients.
In view of (\ref{eqpsiT}) it is not surprising that we will have to
estimate not only the susceptibility
of $\psi_T$ but also of $\phi_T$ and of some Green function with
respect to the random coefficients.
In the first step, we establish the susceptibility estimate for the
Green function. In step 2 we turn to the
susceptibility estimate for the approximate corrector $\phi_T$.
We then show in step 3 that, relying on~\cite{Gloria-Otto-09}, this implies
that the spectral exponents are at least
%
%
\begin{equation}
\label{eqstep-3} \cases{ (2,-q), &\quad for $d=2$,
\vspace*{1pt}\cr
(2,0), &\quad for $d>2$.}
\end{equation}
In step 4 we estimate the susceptibility of $\psi_T$. We conclude the
proof of the proposition in step 5.\vspace*{8pt}

\textit{Step} 1. Susceptibility of the Green function.

We shall prove for all $e=(z,z') \in\mathbb{B}$, $z\in\Z^d$,
$z'=z+\mathbf{e}_i$,
%
%
\begin{eqnarray}
\label{eqstep-1-2}
\frac{\partial G_T}{\partial\omega_e}(x,y) &=& -T^{-1}
\bigl(G_T(z,y)G_T(x,z)+G_T
\bigl(z',y\bigr)G_T\bigl(x,z'\bigr)
\bigr)
\nonumber\\[-8pt]\\[-8pt]
&&{}-\nabla_{2,i}G_T(x,z)\nabla_{1,i}G_T(z,y)
\nonumber
\end{eqnarray}
and
%
%
\begin{eqnarray}
\label{eqstep-1-3}
\sup_{\omega_e}\bigl|\nabla_{1,i}G_T(z,y)\bigr|
&\lesssim& \bigl|\nabla_{1,i}G_T(z,y)\bigr| +T^{-1}g_T(y-z),
\nonumber\\[-8pt]\\[-8pt]
\sup_{\omega_e}\bigl|\nabla_{2,i}G_T(y,z)\bigr|&\lesssim& \bigl|
\nabla_{2,i}G_T(y,z)\bigr| +T^{-1}g_T(y-z),
\nonumber
\end{eqnarray}
where $g_T\dvtx\Z^d\to\R^+$ satisfies for some constant $c>0$ (depending
on $\alpha,\beta,d$)
%
%
\begin{equation}
\label{eqgT-d>2} g_T(x) = \bigl(1+|x|\bigr)^{2-d}\exp\biggl( -c
\frac{|x|}{\sqrt{T}} \biggr)
\end{equation}
for $d>2$, and
%
%
\begin{equation}
\label{eqgT-d=2} g_T(x) = \biggl\llvert\ln\biggl(
\frac{\sqrt{T}}{1+|x|} \biggr)\biggr\rrvert\exp\biggl( -c\frac
{|x|}{\sqrt{T}} \biggr)
\end{equation}
for $d=2$.

We define the elliptic operator $L_T$ as
\[
(L_Tu) (x) = \sum_{x',|x-x'|=1}
\omega_{x,x}T^{-1}u(x)+\sum_{x',|x-x'|=1}
\omega_{x,x'} \bigl(u(x)-u\bigl(x'\bigr) \bigr),\vadjust{\goodbreak}
\]
so that for all $y\in\Z^d$, (\ref{eqdisc-Green}) takes the form
%
%
\begin{equation}
\label{eqstep-1-4} \bigl(L_T G_T(\cdot,y)\bigr) (x) =
\delta(x-y).
\end{equation}
Recalling that the edges are not oriented, a formal differentiation of
this equation with respect to $\omega_e=\omega_{z,z'}=\omega_{z',z}$ yields
\begin{eqnarray*}
&&
L_T \biggl(\frac{\partial G_T}{\partial\omega_e}(\cdot,y) \biggr)
(x)+T^{-1}G_T(x,y)
\bigl(\delta(x-z)+\delta\bigl(x-z'\bigr)\bigr)
\\
&&\quad{}+ \bigl(G_T(z,y)-G_T\bigl(z',y\bigr)
\bigr)\delta(x-z) + \bigl(G_T\bigl(z',y
\bigr)-G_T\bigl(z',y\bigr) \bigr)\delta
\bigl(x-z'\bigr) \\
&&\qquad= 0.
\end{eqnarray*}
Using (\ref{eqstep-1-4}), this identity turns into
\begin{eqnarray*}
&&
L_T \biggl( \frac{\partial G_T}{\partial\omega_e}(\cdot,y) +T^{-1}
\bigl(G_T(z,y)G_T(\cdot,z)+G_T
\bigl(z',y\bigr)G_T\bigl(\cdot,z'\bigr)
\bigr)
\\
&&\hspace*{153.5pt}{} +\nabla_{2,i}G_T(\cdot,z)\nabla_{1,i}G_T(z,y)
\biggr) (x) = 0.
\end{eqnarray*}
Provided that the argument of $L_T$ is well defined (i.e., $G_T$ is
differentiable w.r.t.~$\omega_e$) and that it
is square-integrable on $\Z^d$, it vanishes identically by the Riesz
representation theorem---which is the desired identity (\ref{eqstep-1-2}).

To turn this into a rigorous argument, one may first consider finite
differences of parameter $h>0$ instead of a derivative w.r.t. $\omega
_e$, use that $L_T$ is bijective on the set of square-integrable
functions on $\Z^d$ and then pass to the limit $h\to0$.
We refer the reader to~\cite{Gloria-Otto-09}, proof of Lemma 2.5, for
details, and directly turn to (\ref{eqstep-1-3}).

From (\ref{eqstep-1-2}) with $x=z$ and $x=z'$, we infer that
%
%
\begin{eqnarray}
\label{eqstep-1-2-mod}
&&
\frac{\partial\nabla_{1,i}G_T(z,y)}{\partial
\omega_e} \nonumber\\
&&\qquad= -\nabla_{1,i}
\nabla_{2,i}G_T(z,z) \nabla_{1,i}G_T(z,y)
\\
&&\qquad\quad{} -T^{-1} \bigl( G_T(z,y)\nabla_{1,i}G_T(z,z)+G_T
\bigl(z',y\bigr)\nabla_{1,i}G_T
\bigl(z,z'\bigr) \bigr).
\nonumber
\end{eqnarray}
Using the uniform pointwise estimate of Corollary~\ref{coro}
and the
uniform
estimate on the Green function of Lemma
\ref{lemptwise-Green-decay-opt},
we obtain (\ref{eqstep-1-3}) by considering (\ref{eqstep-1-2-mod})
as an ODE for $\nabla_{1,i}G_T(z,y)$ in function of $\omega_e$.\vspace*{8pt}

\textit{Step} 2. Susceptibility of $\phi_T$.

In this step we shall prove that for $e=(z,z')\in\mathbb
{B}$, $z\in\Z^d$ and $z'=z+\mathbf{e}_i$,
%
%
\begin{eqnarray}
\label{eqstep-2-1}
\frac{\partial\phi_T}{\partial\omega_e}(x) &=& -\bigl
(\nabla_i
\phi_T(z)+\xi_i\bigr)\nabla_{2,i}G_T(x,z)
\nonumber\\[-8pt]\\[-8pt]
&&{}-T^{-1}\phi_T(z) \bigl(G_T(x,z)+G_T
\bigl(x,z'\bigr) \bigr),
\nonumber
\\
\label{eqstep-2-2}
\sup_{\omega_e}\bigl|\phi_T(x)\bigr| &\lesssim&\bigl|
\phi_T(x)\bigr|
\nonumber\\
&&{}+\bigl(\bigl|\nabla_i \phi_T(z)\bigr|+1\bigr) \\
&&\hspace*{11pt}{}\times\bigl(\bigl|
\nabla_{2,i}G_T(x,z)\bigr|+T^{-1/2}g_T(x-z)
\bigr),
\nonumber
\\
\label{eqstep-2-3}\qquad
\sup_{\omega_e}\biggl\llvert\frac{\partial\phi
_T}{\partial\omega_e}(x)\biggr
\rrvert&\lesssim&\bigl(\bigl|\nabla_i \phi_T(z)\bigr|+1\bigr)
\bigl(\bigl|\nabla_{2,i}G_T(x,z)\bigr|+T^{-1/2}g_T(x-z)
\bigr)
\end{eqnarray}
and for all $n\in\N$,
%
%
\begin{eqnarray}
\label{eqstep-2-4}\quad
&&
\sup_{\omega_e}\biggl\llvert\frac{\partial(\phi
_T(x)^{n+1})}{\partial
\omega_e}\biggr
\rrvert\nonumber\\
&&\qquad\lesssim\bigl(\bigl|\nabla_i \phi_T(z)\bigr|+1\bigr)
\bigl(\bigl|\nabla_{2,i}G_T(x,z)\bigr|+T^{-1/2}g_T(x-z)
\bigr)
\\
&&\qquad\quad\hspace*{0pt}{}\times\bigl(\bigl|\phi_T(x)\bigr| +\bigl(\bigl|\nabla_i
\phi_T(z)\bigr|+1\bigr) \bigl(\bigl|\nabla_{1,i}G_T(z,x)\bigr|+T^{-1/2}g_T(x-z)
\bigr)\bigr)^n.
\nonumber
\end{eqnarray}
As for the Green function, we rewrite the defining equation for $\phi
_T$ as
%
%
\begin{equation}
\label{eqstep-2-5} (L_T \phi_T) (x)-\nabla^* \cdot
A(x)\xi= 0.
\end{equation}
Formally differentiating (\ref{eqstep-2-5}) w.r.t. $\omega_e$ yields
\begin{eqnarray*}
&&
L_T\,\frac{\partial\phi_T}{\partial\omega_e}(x)-\bigl(\nabla_i
\phi_T(x)+\xi_i\bigr) \bigl(\delta(x-z)-\delta
\bigl(x-z'\bigr) \bigr)
\\
&&\qquad\hspace*{0pt}{}+T^{-1}\phi_T(x) \bigl(\delta(x-z)+\delta
\bigl(x-z'\bigr) \bigr) = 0,
\end{eqnarray*}
which, using (\ref{eqstep-1-4}), turns into
\begin{eqnarray*}
&&
L_T \biggl( \frac{\partial\phi_T}{\partial\omega_e}-(\nabla_i
\phi_T+\xi_i) \bigl(G_T(
\cdot,z)-G_T\bigl(\cdot,z'\bigr) \bigr)
\\
&&\qquad\hspace*{42.5pt}{}+T^{-1}\phi_T \bigl(G_T(
\cdot,z)+G_T\bigl(\cdot,z'\bigr) \bigr) \biggr) (x) =
0.
\end{eqnarray*}
This (formally) shows (\ref{eqstep-2-1}).

To turn this into a rigorous argument, we may combine (\ref
{eqstep-1-2}) with the Green
representation formula
\[
\phi_T(x) = \int_{\Z^d}G_T(x,y)
\nabla^*\cdot A(y)\xi\,\d y,
\]
which holds since $G_T(x,\cdot)$ is integrable on $\Z^d$ by
Lemma~\ref{lemptwise-Green-decay-opt},
and use standard results of commutation of integration and
differentiation.

We now turn to (\ref{eqstep-2-3}).
This estimate follows from (\ref{eqstep-2-1}), (\ref{eqstep-1-3})
and the following two facts:
%
%
\begin{equation}
\label{eqstep-2-6}
|\phi_T| \lesssim\sqrt{T}
\end{equation}
and
%
%
\begin{equation}
\label{eqstep-2-7} \sup_{\omega_e}\bigl|\nabla_i
\phi_T(z)\bigr| \lesssim\bigl|\nabla_i\phi_T(z)\bigr|+1.
\end{equation}
The starting point to prove (\ref{eqstep-2-6}) is the Green
representation formula in the form of
%
%
\begin{eqnarray}\label{eqstep-2-6-mod}
\bigl|\phi_T(x)\bigr|&=&\biggl\llvert\int_{\Z^d}G_T(x,y)
\nabla^*\cdot A(y)\xi\,\d y\biggr\rrvert
\nonumber
\\
&=&\biggl\llvert\int_{\Z^d}\nabla_2
G_T(x,y)\cdot A(y)\xi\,\d y\biggr\rrvert
\\
&\lesssim& \int_{\Z^d}\bigl|\nabla_2G_T(0,y)\bigr|
\,\d y.\nonumber
\end{eqnarray}
The claim would easily follow if we had the estimate
\[
\bigl|\nabla_2G_T(0,y)\bigr| \lesssim\bigl(1+|y|\bigr)^{1-d}\exp
\biggl(-c\frac
{|y|}{\sqrt{T}} \biggr).
\]
Although this estimate does not hold \textit{pointwise}, it holds when
\textit{square-integrated on dyadic annuli}, as shows
Lemma~\ref{lemint-grad} with ``$p=2$ and $k$ large.'' The claim
(\ref{eqstep-2-6}) thus follows from a dyadic decomposition of space
in (\ref{eqstep-2-6-mod}) combined with Cauchy--Schwarz's inequality
and Lemma~\ref{lemint-grad} (a similar calculation is displayed, e.g.,
in~\cite{Gloria-10}, proof of Lemma 4).

For (\ref{eqstep-2-7}), we first note that (\ref{eqstep-2-1}) implies
\begin{eqnarray*}
\frac{\partial\nabla_i\phi_T(z)}{\partial\omega_e} &=&
-\bigl(\nabla_i\phi_T(z)+ \xi_i\bigr) \bigl( \nabla_{2,i}G_T
\bigl(z',z\bigr)-\nabla_{2,i}G_T(z,z) \bigr)
\\
&&{} +T^{-1}\phi_T(z) \bigl(\nabla_{1,i}G_T(z,z)+
\nabla_{1,i}G_T\bigl(z,z'\bigr) \bigr),
\end{eqnarray*}
which (seen as an ODE w.r.t. $\omega_e$) yields the claim using the
uniform bound $|\nabla_1 G_T|,|\nabla_2G_T|\lesssim1$ of
Corollary~\ref{coro}
and (\ref{eqstep-2-6}).

Estimate (\ref{eqstep-2-2}) is a direct consequence of (\ref
{eqstep-2-3}), whereas
(\ref{eqstep-2-4}) follows from the Leibniz's rule combined with
(\ref{eqstep-2-1}), (\ref{eqstep-2-2}) and (\ref{eqstep-2-3}).\vspace*{8pt}

\textit{Step} 3. Proof of (\ref{eqstep-3}).

The estimates (\ref{eqstep-3}) of the spectral exponents
follow from the more general estimates; for all $q>0$ there
exists $\gamma(q)>0$ such that
%
%
\begin{equation}
\label{eqstep-3-1} \bigl\langle|\phi_T|^q \bigr\rangle
\lesssim\cases{ \ln^{\gamma(q)}T, &\quad for $d=2$,
\cr
1, &\quad for $d>2$,}
\end{equation}
combined with the fact that
\[
\int_0^{T^{-1}} \d e_{\mathfrak{d}}(\lambda)
\lesssim T^{-2}\int_0^\infty
\frac{1}{(T^{-1}+\lambda)^2} \,\d e_{\mathfrak
{d}}(\lambda) = T^{-2} \bigl\langle
\phi_T^2 \bigr\rangle.
\]
The proof of (\ref{eqstep-3-1}) is an easy adaptation of \cite
{Gloria-Otto-09}, proof of Proposition 2.1, which already covers the
case of a constant
coefficient in the zero order term
of $L_T$, that is, for $T^{-1}\phi_T$ instead of $T^{-1}p_\omega\phi_T$
(no randomness in the zero order term).
We only point out what needs to be changed in~\cite{Gloria-Otto-09},
proof of Proposition 2.1.

The first step to apply the variance estimate of Lemma \ref
{lemvar-estim} is to show that $\phi_T$ is measurable
with respect to the cylindrical topology associated with the random
variables. This is proved exactly as in~\cite{Gloria-Otto-09}, Lemma 2.6.

The auxiliary~\cite{Gloria-Otto-09}, Lemmas 2.4 and 2.5, are replaced
by the susceptibility estimates (\ref{eqstep-1-3}), (\ref
{eqstep-2-2}), (\ref{eqstep-2-3})
and (\ref{eqstep-2-4}) of steps 1 and 2, which have, however, the
additional term $T^{-1/2}g_T(x-z)$ next to $|\nabla_{2,i}G_T(x,z)|$.

In the proof of~\cite{Gloria-Otto-09}, Proposition 2.1, the terms
$|\nabla_{2,i}G_T(x,z)|$
are either estimated by the Green function $G_T(x,z)$ itself [in which
case the additional term $T^{-1/2}g_T(x-z)$ is of higher order],
or they are controlled on dyadic annuli by Lemma~\ref{lemint-grad}.
By definition (\ref{eqgT-d>2}) for $d>2$ and (\ref{eqgT-d=2}) for
$d=2$ of the function $g_T$, it is easy to see that
for all $r\geq2$, $k>0$ and $R\gg1$: for $d>2$,
\[
\int_{R\leq|x-z|<2R} \bigl(T^{-1/2}g_T(x-z)
\bigr)^r \,\d z \lesssim R^d\bigl(R^{1-d}
\bigr)^r \min\bigl\{1,\sqrt{T}R^{-1}\bigr\}^k,
\]
whereas for $d=2$
\[
\int_{R\leq|x-z|<2R} \bigl(T^{-1/2}g_T(x-z)
\bigr)^r \,\d z \lesssim R^2 \bigl(R^{-1}
\bigr)^r \ln^qT \min\bigl\{1,\sqrt{T}R^{-1}\bigr
\}^k.
\]
These scalings coincide with those of Lemma~\ref{lemint-grad} (with a
possible additional logarithmic correction for $d=2$).

Hence, the proof of~\cite{Gloria-Otto-09}, Proposition 2.1, adapts
mutatis mutandis to the present case, and we have (\ref{eqstep-3-1}).\vspace*{8pt}

\textit{Step} 4. Susceptibility of $\psi_T$.

In this step we shall prove that for all $e=(z,z')$, $z\in
\Z^d$ and $z'=z+\mathbf{e}_i$ and
for all $x\in\Z^d$,
%
%
\begin{eqnarray}
\label{eqstep-4-1}
\frac{\partial\psi_T}{\partial\omega_e}(x)&=&
-\nabla_{2,i}G_T(x,z)
\nabla_i \psi_T(z)-T^{-1}G_T(x,z)
\psi_T(z)
\nonumber\\
&&{} -T^{-1}G_T\bigl(x,z'\bigr)
\psi_T\bigl(z'\bigr)
\nonumber
\\
&&{} -\bigl(\nabla_i\phi_T(z)+\xi_i\bigr)
\int_{\Z^d}G_T(x,y)p_\omega(y)
\nabla_{2,i}G_T(y,z) \,\d y
\\
&&{} -T^{-1}\phi_T(z)\int_{\Z^d}G_T(x,y)p_\omega(y)
\bigl( G_T(y,z)+G_T\bigl(y,z'\bigr)
\bigr) \,\d y
\nonumber
\\
&&{} +G_T(x,z)\phi_T(z)+G_T
\bigl(x,z'\bigr)\phi_T\bigl(z'\bigr)
\nonumber
\end{eqnarray}
and
%
%
\begin{eqnarray}
\label{eqstep-4-2}
&&
\sup_{\omega_e}\biggl\llvert\frac{\partial\psi
_T}{\partial\omega_e}(x)
\biggr\rrvert
\nonumber\\
\hspace*{5pt}&&\qquad\lesssim g_T(z-x) \bigl( \bigl|\nabla_i
\psi_T(z)\bigr|+T^{-1}\bigl|\psi_T(z)\bigr|+
\nu_d(T) \bigl(1+\bigl|\phi_T(z)\bigr|+\bigl|\phi_T
\bigl(z'\bigr)\bigr|\bigr) \bigr)
\nonumber\\[-8pt]\\[-8pt]
&&\qquad\quad{} +\bigl(1+\bigl|\phi_T(z)\bigr|+\bigl|\phi_T\bigl(z'
\bigr)\bigr|\bigr)\nonumber\\
&&\qquad\quad\hspace*{11pt}{}\times\int_{\Z^d}g_T(y-x) \bigl( \bigl|
\nabla_{2,i}G_T(y,z)\bigr|+T^{-1}g_T(y-z)
\bigr) \,\d y,
\nonumber
\end{eqnarray}
where
%
%
\begin{equation}
\nu_d(T) = \cases{ T, &\quad for $d=2$,
\vspace*{1pt}\cr
\sqrt{T}, &\quad for $d=3$,
\vspace*{1pt}\cr
\ln
T, &\quad for $d=4$,
\vspace*{1pt}\cr
1, &\quad for $d>4$.}
\end{equation}
The starting point is again the Green representation formula
\[
\psi_T(x) = \int_{\Z^d}G_T(x,y)p_\omega(y)
\phi_T(y) \,\d y,
\]
associated with (\ref{eqpsiT}) in the form
\[
T^{-1}p_\omega\psi_T-\nabla^*\cdot A\nabla
\psi_T = p_\omega\phi_T.
\]
Differentiated w.r.t. $\omega_e$ it turns into
\begin{eqnarray*}
\frac{\partial\psi_T(x)}{\partial\omega_e} &=& \int_{\Z^d}\frac
{\partial G_T(x,y) }{\partial\omega_e}
p_\omega(y)\phi_T(y) \,\d y
\\
&&{}+\int_{\Z^d}G_T(x,y)\,\frac{\partial p_\omega(y)}{\partial\omega_e}
\phi_T(y) \,\d y\\
&&{} +\int_{\Z^d}G_T(x,y)
p_\omega(y)\,\frac{\partial
\phi_T(y) }{\partial\omega_e} \,\d y.
\end{eqnarray*}
Combined with (\ref{eqstep-1-2}), (\ref{eqstep-2-1}) and the Green
representation formula itself, this shows (\ref{eqstep-4-1}).

We now turn to (\ref{eqstep-4-2}) and treat each term of the
right-hand side of (\ref{eqstep-4-1}) separately.
We begin with the supremum of the third line of (\ref{eqstep-4-1}),
and claim that
%
%
\begin{eqnarray}
\label{eqline3}
&&
\sup_{\omega_e} \biggl|\bigl(\nabla_i
\phi_T(z)+\xi_i\bigr) \int_{\Z
^d}G_T(x,y)p_\omega(y)
\nabla_{2,i}G_T(y,z) \,\d y \biggr|
\nonumber\\
&&\qquad\lesssim\bigl(1+\bigl|\phi_T(z)\bigr|+\bigl|\phi_T
\bigl(z'\bigr)\bigr|\bigr)\\
&&\qquad\quad{}\times\int_{\Z^d}g_T(y-x)
\bigl( \bigl|\nabla_{2,i}G_T(y,z)\bigr|+T^{-1}g_T(y-z)
\bigr) \,\d y,
\nonumber
\end{eqnarray}
which is proved:
\begin{itemize}
\item using (\ref{eqstep-2-7}) to bound the supremum in $\omega_e$
of $|\nabla_i \phi_T(z)|$ by $|\nabla_i \phi_T(z)|$ itself,
\item bounding $|\nabla_i\phi_T(z)|$ by the triangle inequality
$|\phi_T(z)|+|\phi_T(z')|$,
\item replacing the Green function $G_T$ by $g_T$ using Lemma \ref
{lemptwise-Green-decay-opt},
\item and appealing to (\ref{eqstep-1-3}) to estimate the supremum in
$\omega_e$ of $|\nabla_{2,i}G_T(y,x)|$.
\end{itemize}
This shows that this term is controlled by the second term of the
right-hand side of (\ref{eqstep-4-2}).

The supremum of the term in the fourth line of (\ref{eqstep-4-1})
is also estimated by the second term of the right-hand side of (\ref
{eqstep-4-2}), namely,
%
%
\begin{eqnarray}
\label{eqline4}
&&
\sup_{\omega_e} \biggl|T^{-1}\phi_T(z)\int
_{\Z^d}G_T(x,y)p_\omega(y) \bigl(
G_T(y,z)+G_T\bigl(y,z'\bigr) \bigr) \,\d y
\biggr|
\nonumber\\[-8pt]\\[-8pt]
&&\qquad\lesssim\bigl(1+\bigl|\phi_T(z)\bigr|+\bigl|\phi_T
\bigl(z'\bigr)\bigr|\bigr)T^{-1}\int_{\Z^d}g_T(y-x)g_T(y-z)
\,\d y.
\nonumber
\end{eqnarray}
It is\vspace*{1pt} enough to bound the Green function by $g_T$ using Lemma \ref
{lemptwise-Green-decay-opt},
and to apply (\ref{eqstep-2-2}) for $x=z$ to control $\sup_{\omega
_e}|\phi_T(z)|$,
and use that $|\nabla_1 G_T|,|\nabla_2G_T|,\break T^{-1/2}\*G_T\lesssim1$ by
Corollary~\ref{coro} and Lemma~\ref{lemptwise-Green-decay-opt}.

The suprema of the last two terms of (\ref{eqstep-4-1}) is bounded by
%
%
\begin{eqnarray}
\label{eqline5}
&&
\sup_{\omega_e}\bigl|G_T(x,z)\phi_T(z)+G_T
\bigl(x,z'\bigr)\phi_T\bigl(z'\bigr)\bigr|
\nonumber\\[-8pt]\\[-8pt]
&&\qquad
\lesssim\bigl(1+\bigl|\phi_T(z)\bigr|+\bigl|\phi_T
\bigl(z'\bigr)\bigr|\bigr)g_T(z-x)
\nonumber
\end{eqnarray}
and therefore controlled by the first term of the right-hand side of
(\ref{eqstep-4-2}).
The argument is similar to the proof of (\ref{eqline4}).

The subtle terms are the first three ones, for which we have to
estimate the suprema of $|\nabla_i \psi_T(z)|$, $|\psi_T(z)|$
and $|\psi_T(z')|$ w.r.t. $\omega_e$.

We begin with the following two estimates
%
%
\begin{eqnarray}
\label{equnif-psi-1}
\sup_{\omega_e} \bigl|\psi_T(z)\bigr|&\lesssim& \bigl|\psi_T(z)\bigr|+
\bigl(\bigl|\phi_T(z)\bigr|+\bigl|\phi_T\bigl(z'
\bigr)\bigr|+1 \bigr)\nu_d(T)
\nonumber\\[-8pt]\\[-8pt]
&&{} +\sup_{\omega_e} \bigl|\nabla_i\psi_T(z)\bigr|,
\nonumber
\\
\label{equnif-grad-psi-1}
\sup_{\omega_e} \bigl|\nabla_i\psi_T(z)\bigr|&\lesssim& \bigl|
\nabla_i\psi_T(z)\bigr|+ \bigl(\bigl|\phi_T(z)\bigr|+\bigl|
\phi_T\bigl(z'\bigr)\bigr|+1 \bigr)\nu_d(T)
\nonumber\\[-8pt]\\[-8pt]
&&{} +T^{-1}\sup_{\omega_e} \bigl|\psi_T(z)\bigr|,
\nonumber
\end{eqnarray}
which (seen as a linear system) show that there exists some $T_*>0$
such that for all $T \ge T^*$,
%
%
\begin{eqnarray}
\label{equnif-psi-2}
\sup_{\omega_e} \bigl|\psi_T(z)\bigr|&\lesssim& \bigl|\psi_T(z)\bigr|+
\bigl(\bigl|\phi_T(z)\bigr|+\bigl|\phi_T\bigl(z'
\bigr)\bigr|+1 \bigr)\nu_d(T)
\nonumber\\[-8pt]\\[-8pt]
&&{}+ \bigl|\nabla_i\psi_T(z)\bigr|,
\nonumber
\\
\label{equnif-grad-psi-2}
\sup_{\omega_e} \bigl|\nabla_i\psi_T(z)\bigr|&\lesssim& \bigl|
\nabla_i\psi_T(z)\bigr|+ \bigl(\bigl|\phi_T(z)\bigr|+\bigl|
\phi_T\bigl(z'\bigr)\bigr|+1 \bigr)\nu_d(T)
\nonumber\\[-8pt]\\[-8pt]
&&{}+T^{-1}\bigl|\psi_T(z)\bigr|.
\nonumber
\end{eqnarray}
To prove (\ref{equnif-psi-1}) we consider (\ref{eqstep-4-1}) as an
ODE on $\psi_T(z)$,
bound $\psi_T(z')$ by $\psi_T(z)+|\nabla_i \psi_T(z)|$ and use
(\ref{eqline3}), (\ref{eqline4}) and (\ref{eqline5}) (for $x=z$),
so that (\ref{eqstep-4-1}) turns into
\begin{eqnarray*}
\biggl\llvert\frac{\partial\psi_T}{\partial\omega_e}(z)\biggr\rrvert
&\lesssim&
\sup_{\omega_e}\bigl\{\bigl|\nabla_{2,i}G_T(z,z)\bigr|\bigl|
\nabla_i \psi_T(z)\bigr|\bigr\} +T^{-1}G_T(z,z)\bigl|
\psi_T(z)\bigr|
\\
&&{} +T^{-1}G_T\bigl(z,z'\bigr) \Bigl(\bigl|
\psi_T(z)\bigr|+\sup_{\omega_e}\bigl|\nabla_i
\psi_T(z)\bigr| \Bigr)
\\
&&{} +\bigl(1+\bigl|\phi_T(z)\bigr|+\bigl|\phi_T\bigl(z'
\bigr)\bigr|\bigr)
\\
&&{} \times\biggl(1+\int_{\Z^d}g_T(y-z) \bigl( \bigl|
\nabla_{2,i}G_T(y,z)\bigr|+T^{-1}g_T(y-z)
\bigr) \,\d y \biggr).
\end{eqnarray*}
Using Corollary~\ref{coro} and Lemma~\ref{lemptwise-Green-decay-opt}
in the form of $|\nabla_1 G_T|,|\nabla_2 G_T|,T^{-1}G_T\lesssim1$,
and bounding the gradient of the Green function by
$g_T$ in the integral, we obtain
\begin{eqnarray*}
\biggl\llvert\frac{\partial\psi_T}{\partial\omega_e}(z)\biggr\rrvert
&\lesssim&\sup_{\omega_e}
\bigl\{\bigl|\nabla_i \psi_T(z)\bigr|\bigr\} +\bigl|
\psi_T(z)\bigr| \\
&&{}+\bigl(1+\bigl|\phi_T(z)\bigr|+\bigl|\phi_T
\bigl(z'\bigr)\bigr|\bigr) \biggl(1+\int_{\Z
^d}g_T(y)^2
\,\d y \biggr).
\end{eqnarray*}
Noting that by definitions (\ref{eqgT-d>2}) and (\ref{eqgT-d=2}) of
$g_T$ we have $\int_{\Z^d}g_T(y)^2\,\d y \lesssim\nu_d(T)$,
this inequality turns into
\begin{eqnarray*}
\biggl\llvert\frac{\partial\psi_T}{\partial\omega_e}(z)\biggr\rrvert
&\lesssim&\sup_{\omega_e}
\bigl\{\bigl|\nabla_i \psi_T(z)\bigr|\bigr\} +\bigl|
\psi_T(z)\bigr|\\
&&{} +\nu_d(T) \bigl(1+\bigl|\phi_T(z)\bigr|+\bigl|
\phi_T\bigl(z'\bigr)\bigr|\bigr).
\end{eqnarray*}
Seen as an ODE for $\psi_T$, this implies (\ref{equnif-psi-1}).

We now turn to (\ref{equnif-grad-psi-1}) and infer from (\ref
{eqstep-4-1}) that
\begin{eqnarray*}
\frac{\partial\nabla_i \psi_T(z)}{\partial\omega_e}&=& -\nabla
_{1,i}\nabla_{2,i}G_T(z,z)
\nabla_i \psi_T(z)-T^{-1}\nabla_{1,i}G_T(z,z)
\psi_T(z)
\\
&&{} -T^{-1}\nabla_{1,i}G_T\bigl(z,z'
\bigr)\psi_T\bigl(z'\bigr)
\\
&&{} -\bigl(\nabla_i\phi_T(z)+\xi_i\bigr)
\int_{\Z^d}\nabla_{2,i}G_T(z,y)p_\omega(y)
\nabla_{2,i}G_T(y,z) \,\d y
\\
&&{} -T^{-1}\phi_T(z)\int_{\Z^d}
\nabla_{2,i}G_T(z,y)p_\omega(y) \bigl(
G_T(y,z)+G_T\bigl(y,z'\bigr) \bigr) \,\d y
\\
&&{} +\nabla_{1,i}G_T(z,z)\phi_T(z)+
\nabla_{1,i}G_T\bigl(z,z'\bigr)
\phi_T\bigl(z'\bigr).
\end{eqnarray*}
Repeating the string of arguments leading from (\ref{eqstep-4-1}) to
(\ref{equnif-psi-1}), we deduce (\ref{equnif-grad-psi-1})
and therefore (\ref{equnif-psi-2})
and (\ref{equnif-grad-psi-2}).
Combining the inequality $|\psi_T(z')|\leq|\psi_T(z)|+|\nabla_i
\psi_T(z)|$ with (\ref{equnif-psi-2})
and (\ref{equnif-grad-psi-2}) yields the last estimate we need:
%
%
\begin{equation}\label{equnif-psi-3}
\sup_{\omega_e} \bigl|\psi_T\bigl(z'\bigr)\bigr| \lesssim
\bigl|\psi_T(z)\bigr|+ \bigl(\bigl|\phi_T(z)\bigr|+\bigl|\phi_T
\bigl(z'\bigr)\bigr|+1 \bigr)\nu_d(T) + \bigl|
\nabla_i\psi_T(z)\bigr|.\hspace*{-25pt}
\end{equation}
We are finally in position to conclude the proof of (\ref{eqstep-4-2}).
The last four terms are controlled by (\ref{eqline3}), (\ref
{eqline4}) and (\ref{eqline5}).
Using (\ref{equnif-grad-psi-2}), (\ref{equnif-psi-2}) and (\ref
{equnif-psi-3}), and Corollary~\ref{coro} and Lemma \ref
{lemptwise-Green-decay-opt},
the first three terms of the right-hand side of (\ref{eqstep-4-1})
are controlled by the first term of the right-hand side of (\ref{eqstep-4-2}).
Estimate (\ref{eqstep-4-2}) is proved.\vspace*{8pt}

\textit{Step} 5. Estimate of $\operatorname{var} [\psi_T ]$ for $d>2$
and conclusion.

We apply the variance estimate of Lemma~\ref{lemvar-estim} to $\psi_T$
%
%
\begin{equation}
\label{eqstep-5-bis} \operatorname{var} [\psi_T ] \lesssim\sum
_{e\in
\mathbb{B}} \biggl\langle\sup_{\omega_e}\biggl\llvert
\frac{\partial
\psi_T(0)}{\partial\omega_e} \biggr\rrvert^2 \biggr\rangle
\end{equation}
and appeal to (\ref{eqstep-4-2}).
We distinguish two contributions in this sum and define
\begin{eqnarray*}
A_e&:=& g_T(z) \bigl( \bigl|\nabla_i
\psi_T(z)\bigr|+T^{-1}\bigl|\psi_T(z)\bigr|+
\nu_d(T) \bigl(1+\bigl|\phi_T(z)\bigr|+\bigl|\phi_T
\bigl(z'\bigr)\bigr|\bigr) \bigr),
\\
B_e&:=&\bigl(1+\bigl|\phi_T(z)\bigr|+\bigl|\phi_T
\bigl(z'\bigr)\bigr|\bigr)\int_{\Z^d}g_T(y)
\bigl( \bigl|\nabla_{2,i}G_T(y,z)\bigr|+T^{-1}g_T(y-z)
\bigr) \,\d y.
\end{eqnarray*}
The contribution associated with $A_e$ is estimated as follows:
\begin{eqnarray*}
\sum_{e\in\mathbb{B}} \bigl\langle
A_e^2 \bigr\rangle
&\lesssim& \sum_{e\in\mathbb{B}} \bigl\langle
g_T(z)^2 \bigl( \bigl|\nabla_i
\psi_T(z)\bigr|^2+T^{-2}\bigl|\psi_T(z)\bigr|^2\\
&&\hspace*{51pt}{}+
\nu_d(T)^2 \bigl(1+\bigl|\phi_T(z)\bigr|^2+\bigl|
\phi_T\bigl(z'\bigr)\bigr|^2\bigr) \bigr) \bigr
\rangle
\\
&\lesssim& \biggl(\sum_{z\in\Z^d}g_T(z)^2
\biggr) \bigl( \bigl\langle|\nabla\psi_T|^2 \bigr
\rangle+T^{-2} \bigl\langle\psi_T^2 \bigr
\rangle+\nu_d(T)^2 \bigl(1+ \bigl\langle
\phi_T^2 \bigr\rangle\bigr) \bigr)
\\
&\lesssim& \nu_d(T) \bigl( \bigl\langle|\nabla\psi_T|^2
\bigr\rangle+T^{-2} \bigl\langle\psi_T^2 \bigr
\rangle+\nu_d(T)^2 \bigl(1+ \bigl\langle
\phi_T^2 \bigr\rangle\bigr) \bigr)
\end{eqnarray*}
by stationarity of $\phi_T$, $\psi_T$ and $\nabla\psi_T$.
This is a nonlinear estimate since $ \langle\psi_T^2
\rangle$ and $ \langle|\nabla\psi_T|^2 \rangle$ appear
in the right-hand side whereas we want to estimate
$ \langle\psi_T^2 \rangle$. We then appeal to the
elementary a priori estimate
\[
\bigl\langle|\nabla\psi_T|^2 \bigr\rangle\lesssim
\bigl\langle\phi_T^2 \bigr\rangle^{1/2} \bigl
\langle\psi_T^2 \bigr\rangle^{1/2},
\]
which we obtain by testing (\ref{eqpsiT}) with the test solution
$\psi_T$, integrating by parts, using the bounds on $A$ and
Cauchy--Schwarz's inequality.
Using, in addition, Young's inequality, the estimate turns into
\[
\sum_{e\in\mathbb{B}} \bigl\langle A_e^2
\bigr\rangle-\frac
{1}{C} \bigl\langle\psi_T^2
\bigr\rangle\lesssim C \nu_d(T)^2 \bigl\langle
\phi_T^2 \bigr\rangle+\nu_d(T)^3
\bigl(1+ \bigl\langle\phi_T^2 \bigr\rangle\bigr)
\]
for all $C>0$ and $T$ large enough.\vadjust{\goodbreak}

Combined with (\ref{eqstep-3-1}) for $q=2$ and the definition of $\nu
_d(T)$, this turns into, for all $C>0$,
%
%
\begin{equation}
\label{eqstep-5-3} {\sum_{e\in\mathbb{B}} \bigl\langle
A_e^2 \bigr\rangle} -\frac
{1}{C} \bigl\langle
\psi_T^2 \bigr\rangle
\lesssim C \cases{
T^{3/2}, &\quad for $d=3$,
\cr
\ln^3 T, &\quad for $d=4$,
\cr
1, &\quad for
$d>4$.}
\end{equation}

We now turn to the term associated with $B_e$, which we split into two
terms $B_e=B_{e,1}+B_{e,2}$, where
\begin{eqnarray*}
B_{e,1}&=&\bigl(1+\bigl|\phi_T(z)\bigr|+\bigl|\phi_T
\bigl(z'\bigr)\bigr|\bigr)T^{-1}\int_{\Z
^d}g_T(y)g_T(y-z)
\,\d y,
\\
B_{e,2}&=&\bigl(1+\bigl|\phi_T(z)\bigr|+\bigl|\phi_T
\bigl(z'\bigr)\bigr|\bigr)\int_{\Z^d}g_T(y)
\bigl|\nabla_{2,i}G_T(y,z)\bigr| \,\d y.
\end{eqnarray*}
In particular, we shall prove that
%
%
\begin{equation}
\label{eqB2}
\sum_{e\in\mathbb{B}} \langle B_e
\rangle^2 \lesssim\sum_{e\in\mathbb
{B}} \bigl
\langle B_{e,1}^2 \bigr\rangle+\sum
_{e\in\mathbb
{B}} \bigl\langle B_{2,e}^2 \bigr
\rangle
\lesssim\cases{ T^{3/2}, &\quad for $d=3$,
\vspace*{1pt}\cr
T, &\quad for $d=4$,
\vspace*{1pt}\cr
\sqrt{T}, &\quad for $d=5$,
\vspace*{1pt}\cr
\ln T, &\quad for $d=6$,
\vspace*{1pt}\cr
1, &\quad for $d>6$.}
\end{equation}
We start with the sum of $B_{e,1}^2$ on $\mathbb{B}$.
Since $g_T$ is deterministic and $\phi_T$ is stationary,
\begin{eqnarray*}
\sum_{e\in\mathbb{B}} \bigl\langle B_{e,1}^2
\bigr\rangle&\lesssim&\bigl(1+ \bigl\langle\phi_T^2
\bigr\rangle^{1/2}\bigr) \\
&&{}\times\int_{\Z^d}\int
_{\Z
^d}\int_{\Z^d} T^{-1}g_T(y)g_T
\bigl(y'\bigr)g_T(y-z)g_T
\bigl(y'-z\bigr)\,\d y\,\d y'\,\d z.
\end{eqnarray*}
Using (\ref{eqstep-3-1}) with $q=2$ and the definitions (\ref
{eqgT-d=2}) and (\ref{eqgT-d>2}) of $g_T$ to estimate the integral, we
conclude that the first term of the left-hand side of (\ref{eqB2}) is
controlled by the right-hand side of (\ref{eqB2}). A formal argument to
estimate the triple integral is as follows. By the exponential decay of
$g_T$, it\vspace*{1pt} is enough to integrate on the set
$|y|,|y'|\lesssim\sqrt{T}$ and $|y-z|+|y'-z|\lesssim\sqrt{T}$, and the
integral\vspace*{2pt} essentially behaves as the integral on the ball
of radius $\sqrt{T}$ in $\Z^{3d}$ of $T^{-1}(1+|x|)^{4(2-d)}$, whence
the bounds
\[
\sum_{e\in\mathbb{B}} \bigl\langle
B_{e,1}^2 \bigr\rangle
\lesssim\cases{ T^{3/2},
&\quad for $d=3$,
\vspace*{1pt}\cr
T, &\quad for $d=4$,
\vspace*{1pt}\cr
\sqrt{T}, &\quad for $d=5$,
\vspace*{1pt}\cr
1, &\quad for $d=6$,
\vspace*{1pt}\cr
T^{-1/2}, &\quad for $d=7$,
\vspace*{1pt}\cr
T^{-1}\ln T, &\quad for $d=8$,
\vspace*{1pt}\cr
T^{-1}, &\quad for $d>8$.}
\]
To rigorously prove that the right-hand side of (\ref{eqB2}) is an
upper bound for $\sum_{e\in\mathbb{B}} \langle B_{e,1}^2
\rangle$, we may
simply note that
for $d>2$, if we define $h_T(z):=\sqrt{T}^{-1}\*g_T(z)$, then for all
$z\in\Z^d$,
\[
h_T(z) \lesssim\bigl(1+|z|\bigr)^{1-d}\exp\biggl(-c
\frac{|z|}{\sqrt{T}} \biggr),
\]
and $g_T$ and $h_T$ satisfy the assumptions of Lemma \ref
{lemdb-conv}, which yields the desired upper bound.

We turn to the sum of $B_{e,2}^2$ on $\mathbb{B}$:
\begin{eqnarray*}
{\sum_{e\in\mathbb{B}} \bigl\langle B_{e,2}^2
\bigr\rangle} &\lesssim& \Biggl\langle\sum_{i=1}^d
\int_{\Z^d}\bigl(1+\bigl|\phi_T(z)\bigr|+\bigl|
\phi_T(z+\ee_i)\bigr|\bigr)
\\
&&\hspace*{35pt}{} \times\int_{\Z^d} \int_{\Z^d}
g_T(y) g_T\bigl(y'\bigr)\bigl|
\nabla_{2,i}G_T(y,z)\bigr| \\
&&\hspace*{81.5pt}{}\times\bigl|\nabla_{2,i}G_T
\bigl(y',z\bigr)\bigr|\,\d y \,\d y' \,\d z \Biggr\rangle.
\end{eqnarray*}
Since $g_T$ is deterministic, one can take it out of the expectation.
We then choose $p>2$ such that the higher integrability result of
Lemma~\ref{lemint-grad} applies, and use H\"older's inequality
in probability with exponents $(p/(p-2),p,p)$. By stationarity
of $\phi_T$ and of $G_T$ [in the form of (\ref{eqstat-grad-Green})],
this estimate turns into
\begin{eqnarray*}
{\sum_{e\in\mathbb{B}} \bigl\langle B_{e,2}^2
\bigr\rangle} &\lesssim& \bigl\langle|\phi_T|^{p/(p-2)} \bigr
\rangle^{(p-2)/p} \\
&&{}\times\int_{\Z^d}\int_{\Z^d}
\int_{\Z^d}g_T(y) g_T
\bigl(y'\bigr)
\\
&&\hspace*{62pt}{} \times\bigl\langle\bigl|\nabla_{1}G_T(y-z,0)\bigr|^p
\bigr\rangle^{1/p}\\
&&\hspace*{62pt}{} \times\bigl\langle\bigl|\nabla_{1}G_T
\bigl(y'-z,0\bigr)\bigr|^p \bigr\rangle^{1/p} \,\d y\,\d
y'\,\d z.
\end{eqnarray*}
We then introduce the notation $h_T(x):= \langle|\nabla_{1}G_T(x,0)|^p
\rangle^{1/p}$.
By Lemma~\ref{lemint-grad} and Corollary~\ref{coro}, we have for all
$R\lesssim1$,
\[
\int_{|z|\leq R}h_T(x)^2\,\d x \lesssim1
\]
and for all $R\lesssim1$ and $j\in\N$,
\[
\int_{2^j R\leq|z|< 2^{j+1}R}h_T(x)^2 \,\d x \leq
\biggl(\int_{2^j
R\leq|z|< 2^{i+1}R}h_T(x)^p \,\d x
\biggr)^{2/p} \lesssim\bigl(2^jR\bigr)^{2-d}.
\]
Hence, $g_T$ and $h_T$ satisfy the assumptions of Lemma \ref
{lemdb-conv}, and $\sum_{e\in\mathbb{B}} \langle B_{e,2}^2
\rangle$
is bounded by the right-hand side of (\ref{eqB2}).\vadjust{\goodbreak}

We are in position to estimate $\operatorname{var} [\psi_T^2 ]$.
Choosing $C$ large enough in (\ref{eqstep-5-3}) to absorb the term
$\frac{1}{C}\operatorname{var} [\psi_T^2 ]$ in the
left-hand side of (\ref{eqstep-5-bis}), and using (\ref{eqB2}) we obtain
the estimate
%
%
\begin{equation}
\label{eqestim-psiT-d2}
\operatorname{var} \bigl[\psi_T^2
\bigr] \lesssim\cases{ T^{3/2}, &\quad for $d=3$,
\vspace*{1pt}\cr
T, &\quad for $d=4$,
\vspace*{1pt}\cr
\sqrt{T}, &\quad for $d=5$,
\vspace*{1pt}\cr
\ln T, &\quad for $d=6$,
\vspace*{1pt}\cr
1, &\quad for $d>6$.}
\end{equation}

We may conclude the proof.
Estimate (\ref{eqstep-3}) proved in step 3 yields the desired
spectral exponent for $d=2$, whereas the combination of
(\ref{eqestim-psiT-d2}) with (\ref{eqstep-0-0}) and (\ref
{eqstep-0-1}) yields the desired spectral exponents for $d>2$.
\end{pf*}
%
%
\begin{remark}\label{remopt-spex}
The structure of the proof can be summarized as follows:
\begin{longlist}[(a)]
\item[(a)] The starting point is the optimal estimates of $
\langle\phi_T^2 \rangle$ (up to logarithmic correction for $d=2)$.
\item[(b)] The variance estimate applied to $\psi_T$ and combined
with elliptic theory shows there exists a map $F$ such that
\[
\bigl\langle\psi_T^2 \bigr\rangle\leq F\bigl(T, \bigl
\langle|\nabla\psi_T|^2 \bigr\rangle, \bigl\langle
\psi_T^2 \bigr\rangle, \bigl\langle\phi_T^2
\bigr\rangle\bigr).
\]
\item[(c)] By an a priori estimate, $ \langle|\nabla
\psi_T|^2 \rangle\lesssim\langle\psi_T^2
\rangle^{1/2} \langle\phi_T^2 \rangle^{1/2}$.
\item[(d)] Combined with Young's inequality and \textup{(c)},
\textup{(b)} turns into
\[
\bigl\langle\psi_T^2 \bigr\rangle\leq\tilde F\bigl(T,
\bigl\langle\phi_T^2 \bigr\rangle\bigr)
\]
for some map $\tilde F$, and yields the claim.
\end{longlist}
In view of this, a possible strategy to prove optimal scalings
of the spectral exponents in any dimension would be to proceed by induction.
Set $\phi_{1,T}\equiv\phi_T$, and for all $k\geq1$ define $\phi_{k,T}$
as the unique weak solution
to
\[
T^{-1}\phi_{k+1,T}(x)-\frac{1}{p_\omega(x)}\nabla^*\cdot A(x)
\nabla\psi_{k+1,T}(x) = \psi_{k,T}(x),
\]
and apply the strategy described above to obtain optimal bounds on
$ \langle\phi_{k+1,T}^2 \rangle$
assuming optimal bounds on $ \langle\phi_{k,T}^2 \rangle
$ (which would
yield optimal spectral exponents up to dimension $4k-2$---with a
logarithmic correction in dimension $4k-2$).
The main difficulty is to work out a suitable map $F_k$ in step (b).
\end{remark}
%

%
%
\section{The random fluctuations}\label{sec4}
\label{secrandfluc}

In this section, we show that the computable quantity $\hat{A}_n(t)$
defined in (\ref{defhatA}) is a good approximation of $\sigma_t^2$
in the sense that its random fluctuations are small as soon as $n/t^2$
is large. We write $\N^*$ for $\N\setminus\{0\}$.
%
%
\begin{theorem}
\label{randfluc}
There exists $c > 0$ such that, for any $n \in\N^*$, $\eps> 0$ and
$t$ large enough,
\[
\P^\otimes_0 \bigl[ \bigl|\hat{A}_n(t) -
\sigma_t^2 \bigr| \ge\eps/t \bigr] \le\exp\biggl( -
\frac{n \eps^2}{c t^2} \biggr).
\]
\end{theorem}
Note that $\sigma_t^2$ is the mean value of $\hat{A}_n(t)$, and
moreover, $\hat{A}_n(t)$ consists of a sum of i.i.d. random
variables. We will thus obtain Theorem~\ref{randfluc} by using
classical techniques from large deviation theory. The important point
is that the i.i.d. random variables under consideration are uniformly
exponentially integrable. To see this, we use a sharp upper bound on
the transition probabilities of the random walk recalled in the
following theorem. We refer the reader to \cite
{Hebisch-Saloff-Coste-93} or~\cite{Woess-10}, Theorem~14.12, for a proof.
%
%
\begin{theorem}
\label{gauss}
There exists a constant $c_1 > 0$ such that, for any environment
$\omega$ with conductances in $[\alpha,\beta]$, any $t \in\N^*$
and $x \in\Z^d$,
\[
\PPo_0\bigl[Y(t) = x\bigr] \le\frac{c_1}{t^{d/2}} \exp\biggl(-
\frac
{|x|^2}{c_1 t} \biggr).
\]
\end{theorem}
From Theorem~\ref{gauss} we deduce the following result.
%
%
\begin{cor}
\label{supexp}
Let $c_1$ be given by Theorem~\ref{gauss}. For all $\lambda<
1/c_1$, one~has
\[
\sup_{t \in\N^*} \tilde{\E}_0 \biggl[\exp\biggl(
\lambda\frac
{|Y(t)|^2}{t} \biggr) \biggr] < + \infty.
\]
\end{cor}
\begin{pf}
Let $\delta= 1/c_1 - \lambda$. By Theorem~\ref{gauss},
\[
\EEo_0 \bigl[ e^{ \lambda|Y(t)|^2/t} \bigr] \le c_1
{t^{-d/2}} \sum_{x \in\Z^d} e^{-\delta|x|^2/t}.
\]
If the sum ranges over all $x \in(\N^*)^d$, it is easy to bound it by
a convergent integral,
\[
{t^{-d/2}} \sum_{x \in(\N^*)^d} e^{-\delta|x|^2/t}
\le{t^{-d/2}} \int_{\R_+^d} e^{-\delta|x|^2/t} \,\d x =
\int_{\R_+^d} e^{-\delta
|x|^2} \,\d x.
\]
By symmetry, the estimate carries over to the sum over all $x
\in(\Z^*)^d$. The same argument applies for the sum over all
$x=(x_1,\ldots,x_d)$ having exactly one component equal to $0$, and so
on.
\end{pf}
The following lemma shows that the log-Laplace transform of $\frac
{(\xi\cdot Y(t))^2}{t} - \sigma_t^2$ is bounded by a parabola in a
neighborhood of $0$, uniformly over $t$.
%
%
\begin{lem}
\label{grandesdev}
There exist $\lambda_1 > 0$ and $c_2$ such that, for any $\lambda<
\lambda_1$ and any $t \in\N^*$,
\[
\ln\tilde{\E}_0 \biggl[ \exp\biggl( \lambda\biggl(
\frac{(\xi\cdot
Y(t))^2}{t} - \sigma_t^2 \biggr) \biggr) \biggr]
\le c_2 \lambda^2.
\]
\end{lem}
\begin{pf}
It is sufficient to prove that there exists $c_3$ such that, for any
$\lambda$ small enough and any $t$,
\[
\tilde{\E}_0 \biggl[ \exp\biggl( \lambda\biggl(\frac{(\xi\cdot
Y(t))^2}{t} -
\sigma_t^2 \biggr) \biggr) \biggr] \le1+c_3
\lambda^2.
\]
We use the series expansion of the exponential to rewrite this
expectation as
\[
\sum_{k = 0}^{+\infty} \frac{\lambda^k}{k!} \tilde{
\E}_0 \biggl[ \biggl(\frac{(\xi\cdot Y(t))^2}{t} - \sigma_t^2
\biggr)^k \biggr].
\]
The term corresponding to $k=0$ is equal to $1$, whereas the term for
\mbox{$k = 1$} vanishes.
The remaining sum, for $k$ ranging from $2$ to infinity, can be
controlled using Corollary~\ref{supexp} combined with the bound
\[
\tilde{\E}_0\biggl\llvert\frac{(\xi\cdot Y(t))^2}{t} - \sigma_t^2
\biggr\rrvert^k \le2 \tilde{\E}_0 \biggl[
\frac{(\xi\cdot Y(t))^{2k}}{t^k} \biggr],
\]
which follows from the definition of $\sigma_t^2$ and Jensen's inequality.
\end{pf}
We are now in position to prove Theorem~\ref{randfluc}.
\begin{pf*}{Proof of Theorem~\ref{randfluc}}
From the definition of $\tilde{\P}$ given in (\ref{deftdP}), we can write
\[
\P^\otimes_0 \bigl[\hat{A}_n(t) -
\sigma_t^2 \ge\eps/ t \bigr] = \tilde{\P}^\otimes_0
\biggl[\frac{(\xi\cdot Y^{(1)}(t))^2 + \cdots+
(\xi\cdot Y^{(n)}(t))^2}{n t} - \sigma_t^2 \ge\eps/ t
\biggr].
\]
Let $\lambda> 0$. We bound the latter probability using Chebyshev's inequality,
%
%
\begin{eqnarray}
\label{comput1}
&&\P^\otimes_0 \bigl[\hat{A}_n(t)
- \sigma_t^2 \ge\eps/ t \bigr]
\nonumber
\\
&&\qquad \le\tilde{\E}^\otimes_0 \biggl[ \exp\biggl( \lambda\biggl(
\frac{(\xi
\cdot Y^{(1)}(t))^2 + \cdots+ (\xi\cdot Y^{(n)}(t))^2}{t} - n \sigma
_t^2 \biggr) \biggr) \biggr]\nonumber\\[-8pt]\\[-8pt]
&&\qquad\quad\hspace*{0pt}{}\times
\exp\biggl( -\frac{n \lambda\eps
}{t} \biggr)
\nonumber\\
&&\qquad \le{\tilde{\E}_0 \biggl[ \exp\biggl( \lambda\biggl(
\frac
{(\xi\cdot Y(t))^2}{t} - \sigma_t^2 \biggr) \biggr)
\biggr]}^n \exp\biggl( -\frac{n \lambda\eps}{t} \biggr).\nonumber
\end{eqnarray}
By Lemma~\ref{grandesdev}, the right-hand side of (\ref{comput1}) is
bounded by
\[
\exp\biggl( n \biggl( c_2 \lambda^2 -
\frac{ \lambda\eps}{t} \biggr) \biggr)
\]
for all $\lambda$ small enough. Choosing $\lambda= \eps/2 c_2 t$
(which is small enough for $t$ large enough), we obtain
%
%
\begin{equation}
\label{estim1} \P^\otimes_0 \bigl[\hat{A}_n(t)
- \sigma_t^2 \ge\eps/ t \bigr] \le\exp\biggl( -
\frac{n \eps^2}{4 c_2 t^2} \biggr).
\end{equation}
The probability of the symmetric event
\[
\P^\otimes_0 \bigl[\sigma_t^2 -
\hat{A}_n(t) \ge2 \eps/ t \bigr]
\]
can be handled the same way, so the proof is complete.
\end{pf*}
%
%
\section{Central limit theorem}\label{sec5}
\label{secclt}

In this short section, we complete the analysis by showing that the
quantity $\sqrt{n(t)} ( \hat{A}_{n(t)}(t) - \sigma_t^2
) $ satisfies a
central limit theorem.
%
%
\begin{prop}
\label{pclt}
Let $(n(t))_{t \in\N^*}$ be any sequence tending to infinity with
$t$. Under the measure $\P^\otimes_0$ and as $t$ tends to infinity,
the random variable
\[
\sqrt{n(t)} \bigl( \hat{A}_{n(t)}(t) - \sigma_t^2
\bigr)
\]
converges in distribution to a Gaussian random variable of variance
\[
\mathsf{v} = \biggl( 3 \frac{\E[p^2]}{\E[p]^2} - 1 \biggr) \sigma^4.
\]
\end{prop}
\begin{pf}
Let us define
\[
\V(t) = \frac{p(\omega)(\xi\cdot Y_t)^2}{t \E[p]} - {\sigma_t^2}
\]
and
\[
\V^{(k)}(t) = \frac{p(\omega^{(k)})(\xi\cdot Y_t^{(k)})^2}{t \E
[p]} - {\sigma_t^2},
\]
so that
\[
\hat{A}_{n(t)}(t) - \sigma_t^2 =
\frac{1}{n(t)}\sum_{k = 1}^{n(t)}
\V^{(k)}(t).
\]
Let also\vspace*{1pt} $\mathsf{v}_t = \E_0[\V(t)^2]$. Note that for any $t$, $(\V
^{(k)}(t))_{k \in\N}$ are i.i.d.
centered random variables under $\P_0^\otimes$. From the
Lindeberg--Feller theorem (see, e.g.,~\cite{durr}, Theorem~2.4.5), we know that in order to show
\[
\frac{\sum_{k = 1}^{n(t)} \V^{(k)}(t)}{\sqrt{n(t)}} \,\xxrightarrow{t \to+
\infty}^{(\mathrm{distr.})}\,
\operatorname{Gaussian}(0,\mathsf{v}),
\]
it suffices to check that
%
%
\begin{equation}
\label{cond1}
\mathsf{v}_t \xrightarrow{t \to+ \infty} \mathsf{v}
\end{equation}
and that for any $\eps> 0$,
%
%
\begin{equation}
\label{cond2} \E_0 \bigl[ \V(t)^2
\one_{\{ \V(t) \ge\eps\sqrt{n(t)} \}} \bigr] \xrightarrow{t \to+ \infty
} 0.
\end{equation}
We learn from~\cite{sid} that for almost every environment and as $t$
tends to infinity, $\xi\cdot Y_t/\sqrt{t}$ converges in distribution
under $\PPo_0$ to a Gaussian random variable of variance $\sigma^2$,
that we write $\sigma\mathsf{G}$, where $\mathsf{G}$ is a standard
Gaussian random variable. In order to justify that for almost every
environment, $\xi\cdot Y_t/\sqrt{t}$ converges in distribution to
$(\sigma\mathsf{G})^2$, we need some uniform integrability property,
since the square function is unbounded. But this uniform integrability
is a direct consequence of Theorem~\ref{gauss}. Hence, under $\P_0$
and as $t$ tends to infinity, the random variable
%
%
\begin{equation}
\label{rand3}
\frac{p(\omega)(\xi\cdot Y_t)^2}{t \E[p]}
\end{equation}
converges in distribution to
\[
\frac{p(\omega)}{\E[p]} (\sigma\mathsf{G})^2,
\]
where $\om$ follows the distribution $\P$, and is independent of
$\mathsf{G}$. For the foregoing reason, the squares of the random
variables in (\ref{rand3}) are uniformly integrable as $t$ varies.
Since we know, moreover, that $\lim_{t \to+ \infty} \sigma_t^2 =
\sigma^2$, we thus obtain
\[
\lim_{t \to+ \infty} \mathsf{v}_t = \E\biggl[ \biggl(
\frac
{p(\omega)}{\E[p]} (\sigma\mathsf{G})^2 - \sigma^2
\biggr)^2 \biggr].
\]
We obtain (\ref{cond1}) by expanding this expectation, recalling that
the fourth moment of $\mathsf{G}$ is equal to $3$.

Similarly, Theorem~\ref{gauss} gives us sufficient control to
guarantee that (\ref{cond2}) holds, so the proof is complete.
\end{pf}

%
\section{Numerical validation and comments}\label{sec6}
\label{secnumtest}

In this section, we illustrate on a simple two-dimensional example the
sharpness of the estimates of the systematic error and of the random
fluctuations
obtained in Theorems~\ref{syserr} and~\ref{randfluc}.

In the numerical tests, each conductivity of $\mathbb{B}$ takes the
value $\alpha=1$ or \mbox{$\beta=4$} with probability $1/2$. In this simple
case, the homogenized matrix is given by Dykhne's formula, namely,
$\Ah=\sqrt{\alpha\beta}\mathrm{Id}=2\mathrm{Id}$ (see, e.g.,
\cite{Gloria-10}, Appendix A). For the simulation of the random walk,
we generate (and store) the environment along the trajectory of the
walk. In particular, this requires us to store up to a constant times
$t$ data. In terms of computational cost, the expensive part of the
computations is the generation of the randomness. In particular, to
compute one realization of $\hat A_{t^2}(t)$ costs approximately the
generation of $t^2\times4t=4t^3$ random variables. A~natural advantage
of the method is its full scalability; the $t^2$ random walks used to
calculate a realization of $\hat A_{t^2}(t)$ are completely
independent.

We first test the estimate of the systematic error: up to a logarithmic
correction, the convergence is proved to be linear in time.
In view of Theorem~\ref{randfluc}, typical fluctuations of $t(\hat
A_{n(t)}(t) - \sigma_t^2)$ are of order no greater than $t/\sqrt
{n(t)}$, and thus become negligible when compared with the systematic
error as soon as the number $n(t)$ of realizations satisfies $n(t) \gg t^2$.
We display in Table~\ref{tabsyst} an\vspace*{1pt} estimate of the systematic error
$t\mapsto|\Ah-\frac{\mathbb{E}[p]}{2}\hat A_{n(t)}(t)|$ obtained
with $n(t)=K(t)t^2$ realizations.
The systematic error is plotted on Figure~\ref{figsyst} in function
of the time in logarithmic scale (crosses).
It matches quite well the function $f\dvtx t\mapsto C t^{-1}\ln t$ [for
$C>0$ chosen so that $f(1280)=|\Ah-\frac{\mathbb{E}[p]}{2}\hat
A_{n(640)}(1280)|$] which is plotted as a solid line.
This is consistent with Theorem~\ref{syserr} and supports the fact
that the spectral exponents are $(2,0)$ for $d=2$ [and not $(2,-q)$ for
some $q>0$].

%
%
\begin{table}[t]
\tabcolsep=0pt
\caption{Systematic error $|\Ah-\frac{\mathbb{E}[p]}{2}\hat
A_{n(t)}(t)|$ in function of the final time $t$ for $K(t)t^2$
realizations}\label{tabsyst}
\begin{tabular*}{\tablewidth}{@{\extracolsep{\fill}}lcccccccc@{}}
\hline
$t$ & 10 & 20 & 40 & 80 & 160 &320 &640 &1280\\
[4pt]
$K(t)$ & $10^4$ & $10^4$ & $10^4$ & $10^4$ & $10^4$ & $10^4$ & $4.0
10^3$ & $10^3$
\\[4pt]
Systematic  & 9.27E--02 & 5.31E--02 & 3.09E--02 & 1.71E--02 &
9.58E--03 & 5.45E--03 & 2.93E--03 & 1.66E--03\\
\quad error\\
\hline
\end{tabular*}
\end{table}

%
%
\begin{figure}[b]

\includegraphics{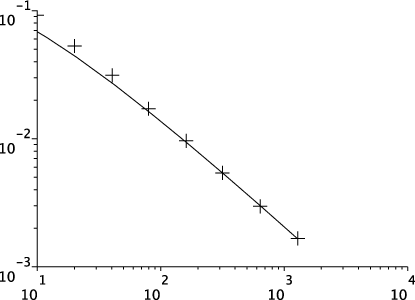}

\caption{Systematic error $|\Ah-\frac{\mathbb{E}[p]}{2}\hat
A_{n(t)}(t)|$ in function of the final time $t$ for $n(t)=K(t)t^2$
realizations (logarithmic scale).}
\label{figsyst}
\end{figure}

We now turn to the random fluctuations of $\hat A_{n(t)}(t)$.
Theorem~\ref{randfluc} gives us a Gaussian upper bound on the tail of
the fluctuations of $t(\hat A_{n}(t)-\sigma_t^2)$, measured in units
of $t/\sqrt{n}$,
whereas Proposition~\ref{pclt} proves the corresponding central limit
theorem, that is, convergence in distribution of $t(\hat A_{t^2}(t) -
\sigma_t^2)$ to a Gaussian random variable.
The Figures~\ref{fighisto-1}--\ref{fighisto-6} display the
histograms of $t \frac{\E[p]}{2} (\hat A_{t^2}(t) - \sigma_t^2)$
for $t=10,20,40$ and $80$ [with 10,000 realizations of $\hat A_{t^2}(t)$
in each case, and $\sigma_t^2$ approximated by
the empirical mean of $\hat A_{t^2}(t)$ over the 10,000 realizations].
As expected, they look Gaussian.
In addition, Proposition~\ref{pclt} also gives the limiting variance.
Table~\ref{tabvar} displays the limiting variance $(\E[p]/2)^2
\mathsf{v}=9.08$ and the empirical variances for $t=10,20,40,80,160$
and $320$,
which are in good agreement.

%
%
\begin{figure}

\includegraphics{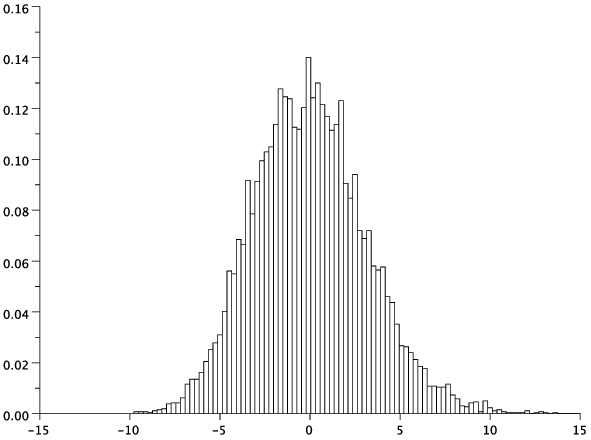}

\caption{Histogram of the re-scaled fluctuations for $t=10$.}
\label{fighisto-1}
\end{figure}

%
\begin{figure}

\includegraphics{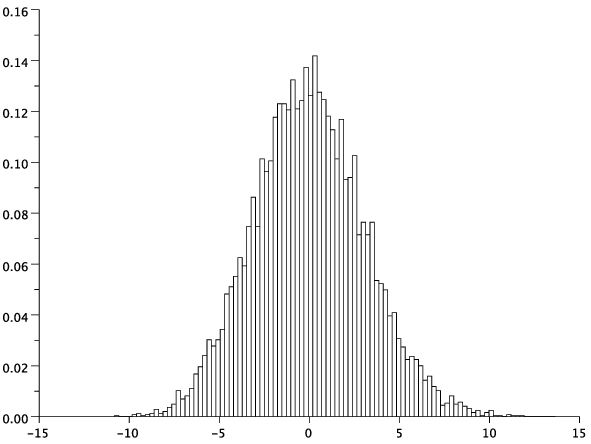}\vspace*{-3pt}

\caption{Histogram of the re-scaled fluctuations for $t=20$.}
\label{fighisto-2}
\end{figure}

%
\begin{figure}[b]

\includegraphics{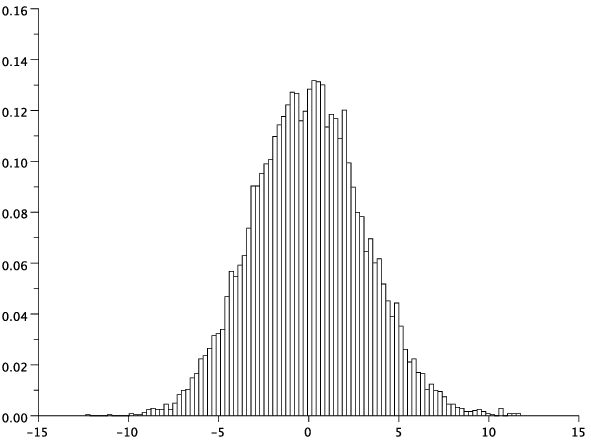}\vspace*{-3pt}

\caption{Histogram of the re-scaled fluctuations for $t=40$.}
\label{fighisto-3}
\end{figure}
%

%
%
\begin{figure}

\includegraphics{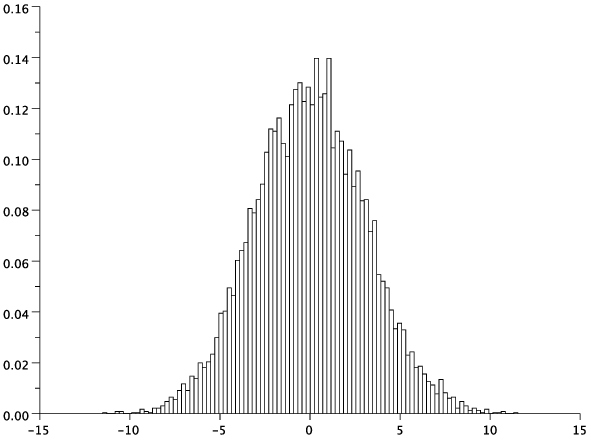}\vspace*{-3pt}

\caption{Histogram of the re-scaled fluctuations for $t=80$.}
\label{fighisto-4}
\end{figure}

%
\begin{figure}[b]

\includegraphics{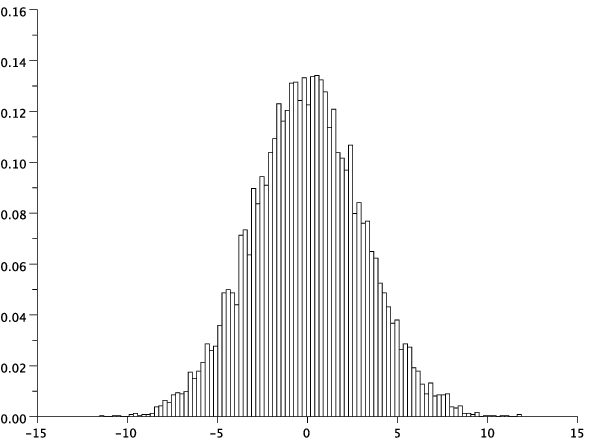}\vspace*{-3pt}

\caption{Histogram of the re-scaled fluctuations for $t=160$.}
\label{fighisto-5}
\end{figure}

%
\begin{figure}

\includegraphics{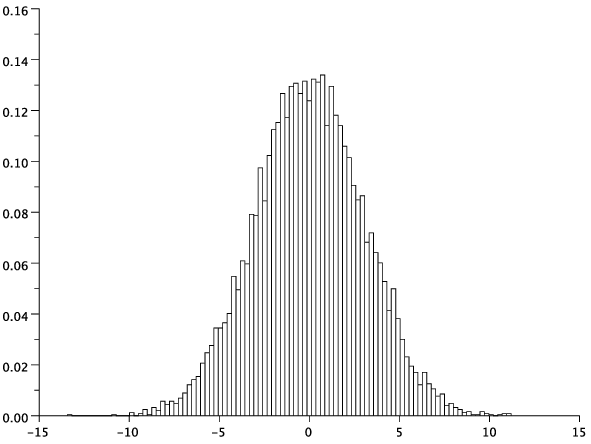}

\caption{Histogram of the re-scaled fluctuations for $t=320$.}
\label{fighisto-6}
\end{figure}
To conclude this article, let us quickly compare the Monte Carlo
approach under consideration here to other approaches to approximate
homogenized coefficients.
Another possibility to approximate effective coefficients is to
directly solve the so-called corrector equation.
In this approach, a first step toward the derivation of error estimates
is a quantification of the qualitative results proved by
K\"unnemann~\cite{Kunnemann-83}
(and inspired by Papanicolaou and Varadhan's treatment of the
continuous case~\cite{Papanicolaou-Varadhan-79}) and Kozlov~\cite{Kozlov-87}.
In the stochastic case, such
an equation is posed on the whole $\Z^d$, and we need to localize it
on a bounded domain, say the hypercube $Q_R$ of side $R>0$.
As shown in a series of papers by Otto and the first author \cite
{Gloria-Otto-09,Gloria-Otto-09b}, and the first author~\cite{Gloria-10},
there are three contributions to the $\LL^2$-error in probability
between the true homogenized coefficients and its approximation.
The dominant error in small dimensions takes the form of a variance: it
measures the fact that the approximation of the homogenized coefficients
by the average of the energy density of the corrector on a box $Q_R$
fluctuates. This error
decays at the rate of the central limit theorem $R^{-d}$ in any
dimension (with a logarithmic
correction for $d=2$). The second error is a systematic error: it is
due to the fact that we have modified the corrector equation
by adding a zero-order term of strength $T^{-1}>0$ (as is standard in
the analysis of the well-posedness of the corrector equation).
The scaling of this error depends on the
dimension and saturates at dimension $4$. It is of higher order than
the random error up to dimension $8$. The last error is due to the use
of boundary conditions on the bounded domain $Q_R$. Provided there is a
buffer region, this error is exponentially small in the distance to the buffer
zone measured in units of $\sqrt{T}$.

%
\begin{table}[b]
\caption{Empirical variance of $\frac{\mathbb{E}[p]}{2} t(\hat
A_{t^2}(t) -\sigma_t^2)$ and limiting variance from
Proposition \protect\ref{pclt}}\label{tabvar}
\begin{tabular*}{\tablewidth}{@{\extracolsep{\fill}}lccccccc@{}}
\hline
$t$ & 10 & 20 & 40 & 80 & 160 & 320 & $\infty$
\\[4pt]
Variance & 9.86 & 9.46 & 9.49 & 9.46 & 9.36 & 9.06 & 9.08\\
\hline
\end{tabular*}
\end{table}

This approach has two main drawbacks. First, the numerical method only
converges at the central limit theorem (CLT) scaling in terms of $R$ up
to dimension
$8$, which is somehow disappointing from a conceptual point of view
(although this is already fine in practice). Second, although the size
of the buffer zone is roughly independent of the dimension, its cost
with respect to the central limit theorem scaling dramatically
increases with the
dimension (recall that in dimension $d$, the CLT scaling is $R^{-d}$,
so that in high dimension, we may consider smaller $R$ for a given precision,
whereas the use of boundary conditions requires $R\gg\sqrt{T}$ in any
dimension).
Based on ideas of the second author in~\cite{Mourrat-10}, we have
taken advantage of the spectral representation of the homogenized coefficients
(originally introduced by Papanicolaou and Varadhan to prove their
qualitative homogenization result) in order to devise and analyze new
approximation formulas for the homogenized coefficients in \cite
{Gloria-Mourrat-10}. In particular, this has allowed us to get rid of
the restriction
on dimension, and exhibit refinements of the numerical method of \cite
{Gloria-10} which converge at the central limit theorem scaling in any
dimension (thus avoiding the first mentioned drawback).
Unfortunately, the second drawback is inherent to the type of method
used: if the corrector equation has to be solved
on a bounded domain $Q_R$, boundary conditions need to be imposed on
the boundary $\partial Q_R$.
Since their values are actually also part of the problem, a buffer zone
seems mandatory---with the notable exception of the periodization
method, whose analysis is
yet still unclear to us, especially when spatial correlations are
introduced in the coefficients.

In this paper we have analyzed a method which does not suffer from
the drawbacks mentioned above: the random
walk in random environment approach.
In particular, following~\cite{Papanicolaou-83} we have obtained an
approximation of the homogenized coefficients by the numerical
simulation of a random walk up to some large time.
Compared to the deterministic approach based on the approximate corrector
equation, the advantage of the present approach is that its convergence
rate and computational costs are dimension-independent.
In addition, the environment only needs to be generated along the trajectory
of the random walker, so that much less information has to be stored
during the calculation. This may be quite an important feature of the
Monte Carlo method
in view of the discussion of~\cite{Gloria-10}, Section 4.3.

A more thorough comparison of these numerical approaches in two and
three dimensions, for correlated and uncorrelated examples, will be the object
of a forthcoming work~\cite{EGMN}.

%
%


%

\printaddresses

\end{document}